\documentclass{amsart}

\usepackage{tikz}
\usepackage{graphicx}
\usepackage{amssymb}
\usepackage[latin1]{inputenc} 
\usepackage{hyperref}

\newtheorem{theorem}{Theorem}[section]
\newtheorem*{mydef}{Theorem A}
\newtheorem*{mydef2}{Theorem D}
\newtheorem*{mydef3}{Corollary A}
\newtheorem*{mydef4}{Theorem B}
\newtheorem*{mydef5}{Theorem C}
\newtheorem*{mydef6}{Theorem E}
\newtheorem{lemma}[theorem]{Lemma}
\newtheorem{conj}[theorem]{Conjecture}
\newtheorem*{claim}{Claim}

\theoremstyle{definition}
\newtheorem{definition}[theorem]{Definition}
\newtheorem{example}[theorem]{Example}

\theoremstyle{remark}
\newtheorem{remark}[theorem]{Remark}
\theoremstyle{proposition}
\newtheorem{proposition}{Proposition}
\theoremstyle{Corollary}
\newtheorem{corollary}{Corollary}

\numberwithin{equation}{section}

	\def\Diff{\operatorname{Diff}}
	\def\inte{\operatorname{int}}
	\def\L{\operatorname{Leb}}
       	\def\Oh{\operatorname{O}}

\begin{document}
	
\title{Phase Transitions on the Markov and Lagrange Dynamical Spectra}

\author{Davi Lima}

\address{Davi Lima: Av Lourival Melo Mota s/n}
\curraddr{Instituto de Matem\'atica, Universidade Federal de Alagoas-Brazil}

\email{davimat@impa.br}
\thanks{The first author was partially supported by CAPES and CNPq.}


\author{Carlos Gustavo Moreira}
\address{Carlos Gustavo Moreira: School of Mathematical Sciences,   
Nankai University, Tianjin 300071, P. R. China, and 
IMPA, Estrada Dona Castorina 110, 22460-320, Rio de Janeiro, Brazil
}
\email{gugu@impa.br}
\thanks{The second author was partially supported by CNPq and Faperj.}

\date{\today}
%    General info
%\subjclass[2000]{Primary 54C40, 14E20; Secondary 46E25, 20C20}

%\date{January 06, 2018 and, in revised form, June 22, 2001.}



\begin{abstract}
The Markov and Lagrange dynamical spectra were introduced by Moreira and share several geometric and topological aspects with the classical ones. However, some features of generic dynamical spectra associated to hyperbolic sets can be proved in the dynamical case and we do not know if they are true in classical case.

They can be a good source of natural conjectures about the classical spectra: it is natural to conjecture that some properties which hold for generic dynamical spectra associated to hyperbolic maps also holds for the classical Markov and Lagrange spectra.

In this paper, we show that, for generic dynamical spectra associated to horseshoes, there are transition points $a$ and $\tilde{a}$ in the Markov and Lagrange spectra respectively, such that for any $\delta>0$, the intersection of the Markov spectrum with $(-\infty,a-\delta)$ 
has Hausdorff dimension smaller than one, while the intersection of the Markov spectrum with $(a,a+\delta)$ has non-empty interior. Similarly, the intersection of the Lagrange spectrum with $(-\infty,\tilde{a}-\delta)$ has Hausdorff dimension smaller than one, while the intersection of the Lagrange spectrum with $(\tilde{a},\tilde{a}+\delta)$ has non-empty interior. We give an open set of examples where $a\neq \tilde{a}$ and we prove that, in the conservative case, generically, $a=\tilde{a}$ and, for any $\delta>0$, the intersection of the Lagrange spectrum with $(a-\delta, a)$ 
has Hausdorff dimension one.
\end{abstract}

\maketitle

\tableofcontents
\section{Introduction}

%% The correct journal style for \specialsection is all uppercase; a known bug
%% in amsart.cls prevents this, so input must be uppercase until it is fixed.
%\specialsection*{This is a Special Section Head}
%\specialsection*{THIS IS A SPECIAL SECTION HEAD}
%This is an example of a special section head%
%%%%%%%%%%%%%%%%%%%%%%%%%%%%%%%%%%%%%%%%%%%%%%%%%%%%%%%%%%%%%%%%%%%%%%%%
%\footnote{Here is an example of a footnote. Notice that this footnote
%text is running on so that it can stand as an example of how a footnote
%with separate paragraphs should be written.
%\par
%And here is the beginning of the second paragraph.}%
%%%%%%%%%%%%%%%%%%%%%%%%%%%%%%%%%%%%%%%%%%%%%%%%%%%%%%%%%%%%%%%%%%%%%%%%
.

\subsection{Classical Markov and Lagrange spectra from Number Theory}\
Regular Cantor sets on the real line play a fundamental role in dynamical systems and notably in problems from Number Theory related to diophantine approximation. They are defined by expansive maps and have some kind of self similarity property: small parts of them are diffeomorphic to big parts with uniformly bounded distortion. Some background on the regular Cantor sets with are relevant to our work can be found in \cite{CF-89}, \cite{PT93}, \cite{MY-01} and \cite{MR-15}.

The classical Lagrange spectrum arises in number theory as the set of finite best constants of irrational numbers when it is approximated by rational numbers. Precisely, given any irrational number $\alpha$ we know by Dirichlet's theorem that the inequality, $|\alpha -p/q|<1/q^2$, has infinitely many rational solutions
$p/q$. Hurwitz and Markov improved this result showing that $$|\alpha -p/q|<1/(\sqrt{5}q^2).$$

\bigskip

Meanwhile, for a fixed irrational $\alpha$, better results can be expected. This leads us to associate to each $\alpha$ its best constant of approximation (Lagrange value of $\alpha$), given by
\begin{eqnarray*}
	k(\alpha)&=&\sup \left\{k>0;\left|\alpha -\frac{p}{q}\right|<\frac{1}{kq^2} \ \mbox{has infinitely many rational solution} \ \frac{p}{q}\right \}\\&
	=&\limsup_{p\in \mathbb{Z},q\in \mathbb{N}, p,q\to \infty}|q(q\alpha-p)|^{-1}\in \mathbb{R}\cup \{\infty\}.
\end{eqnarray*}

The Hurwitz-Markov's theorem implies that $k(\alpha)\geq \sqrt{5}.$ The set of irrational numbers $\alpha$ such that $k(\alpha)<\infty$ has zero Lebesgue measure, but Hausdorff dimension $1$. Consider the set
$$L=\{k(\alpha); \alpha \in \mathbb{R}-\mathbb{Q}, k(\alpha)<\infty\}.$$

The set $L$ is known as the \textbf{Lagrange spectrum}. It is a closed subset of $\mathbb{R}$ and Markov showed in \cite{M79} the 
$$L\cap (-\infty, 3)=\{k_1=\sqrt{5}<k_2=2\sqrt{2}<k_3=\frac{\sqrt{221}}{5}<...\},$$
where $k^2_n\in \mathbb{Q}$ for every $n\in \mathbb{N}$ and $k_n\to 3$ when $n\to \infty$ (for more properties of $L$, cf. [CF 89]).

In 1947, M. Hall (cf. \cite{H47}) proved that for the regular Cantor set $C(4)$ of real numbers in $[0,1]$ for which they have coefficients $1, 2, 3, 4$ in its continued fraction expansion, we have
$$C(4)+C(4)=\{x+y;x,y\in C(4)\}=[\sqrt{2}-1,4(\sqrt{2}-1)].$$

We can write the continued fraction of $\alpha$ as $\alpha=[a_0;a_1,...]$ and for each $n\in \mathbb{N}$, we put $\alpha_n=[a_n;a_{n+1},...]$ and $\beta_n=[0;a_{n-1}, a_{n-2},...,a_1].$ Using elementary continued fraction techniques it can be showed that 
$$k(\alpha)=\limsup_{n\to \infty}(\alpha_n+\beta_n).$$

With the above characterization of $k(\alpha)$ and from Hall's results, it follows that $L\supset [6,+\infty).$

In 1975, Freiman (cf. \cite{CF-89}) proved some difficult results on the arithmetic sums of regular Cantor sets, related to continued fractions, and using them he showed that the biggest interval contained in $L$ is $[c,+\infty)$, where 
$$c=\frac{2221564096+283748\sqrt{462}}{491993569}\simeq 4,52782956616...$$

There are several characterizations of $L$. We give one that will be useful for dynamical generalizations. Throughout in this paper, we denote $\mathbb{N}=\{1,2,...\}$ the set of natural numbers while $\mathbb{N}_0=\mathbb{N}\cup \{0\}$.

Consider $\Sigma=\mathbb{N}^{\mathbb{Z}}$ and let $\sigma:\Sigma \rightarrow \Sigma$ be the shift map defined by $\sigma((a_n)_{n\in \mathbb{Z}})=(a_{n+1})_{n\in \mathbb{Z}}.$ If we define $f:\Sigma \rightarrow \mathbb{R}$ by 
$$f((a_n)_{n\in \mathbb{Z}})=\alpha_0+\beta_0=[a_0;a_1,...]+[0;a_{-1},a_{-2},...],$$
then 
$$L=\left\{\limsup_{n\to \infty} f(\sigma^n(\underline{\theta})); \underline{\theta}\in \Sigma \right \}.$$

In a similar way we can define the set
\begin{equation}\label{m1}
M=\left\{\sup_{n\in \mathbb{Z}}f(\sigma^n(\underline{\theta})); \underline{\theta}\in \Sigma \right\}.
\end{equation}
The set $M$ is called \textbf{Markov spectrum} and it also has an arithmetical interpretation. In fact, if  $\mathcal{Q}=\{f(x,y)=ax^2+bxy+cy^2; \ f \ \mbox{real indefinite and} \ b^2-4ac=1\}$ and $m(f)=\inf\{|f(x,y)|; (x,y)\in \mathbb{Z}^2\backslash \{(0,0)\} \}$ then (cf. \cite{CF-89}) 
$$M=\{m(f)^{-1}<\infty;f\in \mathcal{Q}\}.$$
Thanks to (\ref{m1}) we have that $L\subset M$. During a long time it was believed that $L=M$. In fact this is not true, Freiman showed that the number $\sigma\approx 3.118120178...$ is in $M$ but not in $L$ ( cf. \cite{CF-89} ). Moreover $L$ and $M$ are closed sets. 
An important open question related to Lagrange and Markov spectra is to know whether $C(2)+C(2)$ contains any interval. S. Astels (cf. \cite{SA99})
showed using a method similar to local thickness that $C(2)+C(5)+\mathbb{Z}=\mathbb{R}$.

C. G. Moreira showed in \cite{M3}
that if we define $d(t)=HD(L\cap(-\infty, t))$ then $d(t)$ is a continuous and surjective function from $\mathbb{R}$ onto $[0,1]$ and 
\begin{equation}\label{m2}
d(t)=HD(M\cap (-\infty,t))
\end{equation}
where $HD(A)$ is the Hausdorff dimension of a set $A\subset \mathbb{R}$. Let $a=\inf\{t\in \mathbb{R}; d(t)=1\}$. It follows from the results in \cite{M3} and the Bumby's results in \cite{B} 
that $3.33437...<a<\sqrt{12}$.  In \cite{M3} Moreira ask the following question:

\textbf{Question 1.} Is it true that for every $\delta>0$ 
\begin{eqnarray}\label{Q1}
\mbox{int}(L\cap (\infty, a+\delta))\neq \emptyset?
\end{eqnarray} 

In the same context we leave the following question:

\textbf{Question 2.} In the negative case for Question 1, what is the smallest $\delta>0$ for which (\ref{Q1}) 
holds? 

\bigskip

We write $C(N)=\{x=[0;a_1,a_2,...]; a_j\le N, \forall \ j\ge 1\}$. Since $HD(C(2))>0.53$, a positive answer for Question 1 implies that $C(2)+C(2)$ contains an interval.
\subsection{Generalized Markov and Lagrange spectra from Dynamical Systems}

The above questions are closely related to Palis' conjecture on the regular Cantor sets. It says that, generically, the arithmetic difference of two regular Cantor sets $K$ and $K'$,
$$K-K'=\{x-y;x\in K, y\in K'\},$$
have zero Lebesgue measure or nonempty interior. In \cite{MY-01} Moreira and Yoccoz proved a stronger version of this conjecture.

%If we can characterize when $K-K'$ has nonempty interior, maybe we could solve a lot of these questions.   

The characterization of $L$ via $(\Sigma, \sigma, f)$ admits a natural generalization in the context of hyperbolic dynamics of surfaces. Throughout the text $M$ denotes a bidimensional surface.  

Let $\varphi:M\rightarrow M$ be a $C^2$ diffeomorphism of a surface $M$ with a compact invariant set $\Lambda$ (for example a horseshoe) and let $f:M\rightarrow \mathbb{R}$ be a continuous function. We define for $x\in M$, \textbf{the Lagrange value} of $x$ associated to $f$ and $\varphi$ as being the number $l_{f,\varphi}(x)=\limsup_{n\to \infty}f(\varphi^n(x))$. Similarly, the \textbf{Markov value} of $x$ associated to $f$ and $\varphi$ is the number $m_{f,\varphi}(x)=\sup_{n\in \mathbb{Z}}f(\varphi^n(x))$. The sets
$$L_f(\Lambda)=\{l_{f,\varphi}(x);x\in \Lambda\}$$
and
$$M_f(\Lambda)=\{m_{f,\varphi}(x);x\in \Lambda\}$$
are called \textbf{Lagrange Spectrum} of $(f,\Lambda)$ and \textbf{Markov Spectrum} of $(f, \Lambda)$.

In a similar way as the classical Lagrange and Markov spectra we can define two real functions
\begin{equation}\label{f1}
d_{L_f(\Lambda)}(t)=HD(L_{f}(\Lambda)\cap (-\infty,t))
\end{equation}
and
\begin{equation}
d_{M_f(\Lambda)}(t)=HD(M_{f}(\Lambda)\cap (-\infty,t)).
\end{equation}

We leave the question

\textbf{Question 3.}\label{Q3}
When $d_{L_f(\Lambda)}(t)=d_{M_f(\Lambda)}(t)$ and is it a continuous functions?

\bigskip

Cerqueira, Matheus and Moreira showed in \cite{CMM16} in the conservative case, i.e., which preserves a smooth measure, that $d_{L_f(\Lambda)}=d_{M_f(\Lambda)}$ and it is a continuous function. The other cases remains open.

Our main goal in this paper is to give a positive answer to question 1 for the Markov and Lagrange spectra associated to generic hyperbolic dynamics. %For this goal we need some preliminary. 

%Recall that if $\{P_a\}_{a\in \mathcal{A}}$ is a Markov partition of $\Lambda$ then we say that $(a_0,a_1)$ is \textbf{admissible} if 
%$$\varphi(P_{a_0})\cap P_{a_1}\neq \emptyset.$$
%Let $\mathcal{B}=\{(a_0,a_1); (a_0,a_1) \ \mbox{is permitted}\}$ be the set of \textbf{admissible} words of length $2$  and let $\mathcal{B}_n=\{\underline{a}=(a_0,a_1,...,a_{n-1});(a_i,a_{i+1}) \in \mathcal{B}, 0\leq i\leq n-2\}$ be the set of \textbf{admissible} words of length $n$.
%Since $\varphi$ is a $C^2$-diffeomorphism we have $C^{1+\epsilon}$-foliations $\mathcal{F}^{s,u}$ and with these foliations we can define $C^{1+\epsilon}$-projections $\pi_a^{u,s}$ such that 
%$$K^{s,u}=\bigcup_{a\in \mathcal{A}}\pi_a^{u,s}(\Lambda\cap R_a)$$
%are $C^{1+\epsilon}$-regular Cantor sets, we call them, respectively, stable and unstable Cantor sets.

Given $f:M\rightarrow \mathbb{R}$ we define 
$$\Lambda_t=\bigcap_{n\in \mathbb{Z}}\varphi^n(\Lambda\cap f^{-1}(-\infty, t]).$$

We can imagine $\Lambda_t$ as the horseshoe under the viewpoint of $f$ up the time $t$, but $\Lambda_t$ may not be a horseshoe. Each continuous function sees the horseshoe at time $t$ in its own way. %The Lagrange dynamical spectrum helps us to locate the $\omega$-limit set of points in $\Lambda$. The sets $\Lambda_t$ are close related with the Markov and Lagrange Spectra.
%By analogy, we define

%$$K^{s,u}_t=\bigcup_{a\in \mathcal{A}}\pi_a^{u,s}(\Lambda_t\cap R_a).$$

Given $\varphi\in \mbox{Diff}^2(M)$ with a horseshoe $\Lambda$, such that $HD(\Lambda)>1$ and $f\in C^1({M,\mathbb{R}})$ we put $a=a(f,\varphi)=\sup\{t\in \mathbb{R}; HD(\Lambda_t)<1\}$.

\begin{mydef}\label{TA}
	Given $\varphi\in \Diff^2(M)$ with a horseshoe $\Lambda$, $HD(\Lambda)>1$ and $k\ge 1$, there are an open set $\mathcal{U}\ni \varphi$ and a residual set  $\mathcal{R}\subset \mathcal{U}$ such that for every $\psi\in \mathcal{R}$ there is a $C^k$-residual subset $H_{\psi}\subset C^k(M,\mathbb{R})$ such that
		$$\L(M_f(\Lambda_{\psi})\cap (-\infty,a-\delta))=0,$$
		but
	$$\inte(M_f(\Lambda_{\psi})\cap (-\infty,a+\delta))\neq \emptyset,$$
	for all $\delta>0$, where $a=a(f,\psi)$.
\end{mydef}

\begin{mydef3} We have that
	$$\sup\{t\in \mathbb{R}; HD(\Lambda_t)<1\}=\inf\{t\in \mathbb{R};\inte(M_f(\Lambda)\cap (-\infty,t+\delta))\neq \emptyset, \delta>0\}.$$
\end{mydef3}

In a similar way we have

\begin{mydef4}\label{LPPT}
	Given $\varphi\in \Diff^2(M)$ with a horseshoe $\Lambda$, $HD(\Lambda)>1$ and $k\ge 1$, there are an open set $\mathcal{U}\ni \varphi$ and a residual set  $\mathcal{R}\subset \mathcal{U}$ such that for every $\psi\in \mathcal{R}$ there are a $C^k$-residual subset $H_{\psi}\subset C^k(M,\mathbb{R})$ and $\tilde{a}:=\tilde{a}(f,\psi)$ such that
	$$\L(L_f(\Lambda_{\psi})\cap (-\infty,\tilde{a}-\delta))= 0$$
	but
	$$\inte(L_f(\Lambda_{\psi})\cap (-\infty,\tilde{a}+\delta))\neq \emptyset,$$
	for all $\delta>0$, where $\tilde{a}=\tilde{a}(f,\psi)$.
\end{mydef4}

From the above theorem we can ask the if $a=\tilde{a}$. To the answer this question we have the following theorem.

\begin{mydef5}\label{TheoremC}
	There is a $C^2$-diffeomorphism $\varphi$ on the sphere, such that there are a $C^2$-open $\mathcal{U}\ni \varphi$ and an open and dense set $\mathcal{U}^{\ast}\subset \mathcal{U}$ such that $a(f,\psi)<\tilde{a}(f,\psi)$ for every $\psi\in \mathcal{U}^{\ast}$ and for a $C^1$-open set of $f:\mathbb{S}^2\rightarrow \mathbb{R}$. In particular, 
	$$\inte (M_f(\Lambda_{\psi})\setminus L_f(\Lambda_{\psi}))\neq \emptyset.$$
\end{mydef5}

Nevertheless we can say generically when $a=\tilde{a}$. Let $\Diff^2_{\ast}(M)\subset \Diff^2(M)$ be the set of $C^2$-conservative diffeomorphisms on $M$. We have the following theorem:

\begin{mydef2}\label{TD}
	Given $\varphi\in \Diff^2_{\ast}(M)$ with a horseshoe $\Lambda$, $HD(\Lambda)>1$ and $k\ge 2$, there are an open set $\Diff^2_{\ast}(M)\supset \mathcal{U}\ni \varphi$ and a residual set  $\mathcal{U}^{\ast \ast}\subset \mathcal{U}$ such that for every $\psi\in \mathcal{U}^{\ast \ast}$ there is a $C^k$-residual subset $\mathcal{R}_{\psi}\subset C^k(M,\mathbb{R})$ such that if $f\in \mathcal{R}_{\psi}$ then
	$$\L(M_f(\Lambda_{\psi})\cap (-\infty,a-\delta))= 0 = \L(L_f(\Lambda_{\psi})\cap (-\infty,a-\delta))$$
	and
	$$\inte(L_f(\Lambda_{\psi})\cap (-\infty,a+\delta))\neq \emptyset \neq \inte(M_f(\Lambda_{\psi})\cap (-\infty,a+\delta)).$$
	for all $\delta>0$, where $a=a(f,\psi)$ defined as above.
	Moreover, 
	$$HD(M_f(\Lambda_{\psi})\cap (-\infty,a)))=HD(L_f(\Lambda_{\psi})\cap (-\infty,a))=1.$$\end{mydef2} 

In a similar way, we can define $b(f,\varphi)=\inf\{t\in \mathbb{R}; HD(\Lambda_t)>0\}$. We also prove the following theorem

\begin{mydef6}\label{TE}
	If $\varphi\in \Diff^2(M)$ has a horseshoe $\Lambda$, $k\ge 1$, there is a residual set $\mathcal{R}_{\Lambda}$ in $C^k(M,\mathbb{R})$ such that for every $f\in \mathcal{R}_{\Lambda}$ 
	$$HD(L_f(\Lambda)\cap (-\infty,b+\delta))>0$$
	and
	$$HD(M_f(\Lambda)\cap (-\infty,b+\delta))>0,$$
	for all $\delta>0$, where $b=b(f,\varphi)$.
\end{mydef6}

We observe that, while the Theorem D is unknown to the classical Markov and Lagrange spectra, the Theorem E is also true in the classical case (cf. \cite{M3}). 

These results are close related to fractal geometry of horseshoe. We begin giving some important results on the fractal geometry of regular Cantor sets and horseshoes which have relevance by yourself.
%an open and dense subset $\mathcal{R}_{\Lambda}\subset C^1(M,\mathbb{R})$ such that for every $f\in \mathcal{R}_{\Lambda}$

\subsection{Organization of article}

In section 2, we show that given a regular Cantor, set we can find,  sufficiently close to it, another regular Cantor with different Hausdorff dimension. We state in section 3, a similar result as the previous section, about horseshoes and hyperbolic set of finite type. In section 4, we remember known results on the Markov and Lagrange spectra and make some remarks which derive on the previous results so that we can use them. In section 5, using the result of the sections 3 and 4 we prove Theorem A. In Section 6, we prove Theorems B, C and D. Finally, in section 7 we prove Theorem E. 

\section{Non-stability of the Hausdorff dimension of regular Cantor sets}

In this section, we prove that given a regular Cantor set $K$, there is, for any $k\in \mathbb{N}$, $C^k$-arbitrarily close to it, another regular Cantor set with different Hausdorff dimension. We begin recalling the definition of regular Cantor sets.

\begin{definition}
	A set $K\subset \mathbb{R}$ is called a $C^s$-\textbf{regular Cantor set}, $s\geq 1$, if there exists a collection $\mathcal{P}=\{I_1,I_2,...,I_r\}$ of compacts intervals and a $C^s$ expanding map $\psi$, defined in a neighbourhood of $\displaystyle \cup_{1\leq j\leq r}I_j$ such that
	
	\begin{enumerate}
		\item[(i)] $K\subset \cup_{1\leq j\leq r}I_j$ and $\cup_{1\leq j\leq r}\partial I_j\subset K$,
		
		\item[(ii)] For every $1\leq j\leq r$ we have that $\psi(I_j)$ is the convex hull of a union of $I_t$'s, for $l$ sufficiently large $\psi^l(K\cap I_j)=K$ and $$K=\bigcap_{n\geq 0}\psi^{-n}(\cup_{1\leq j\leq r}I_j).$$
	\end{enumerate}
	
\end{definition}

We say that $\mathcal{P}$ is a \textbf{Markov partition} for $K$ and we will write, when necessary, $(K,\mathcal{P},\psi)$ to understand that $K$ is defined by $\mathcal{P}$ and $\psi$.

\begin{remark}\label{rr1}
	We gave a definition of regular Cantor set from [12]. There are alternative definitions of regular Cantor sets. For instance, in [10] they defined as follows: 
	
	Let $\mathcal{A}$ be a finite alphabet, $\mathcal{B}\subset \mathcal{A}^2$, and $\Sigma$ the subshift of finite type of $\mathcal{A}^{\mathbb{Z}}$ with allowed transitions $\mathcal{B}$ which is topologically mixing, and such that every letter in $\mathcal{A}$ occurs in $\Sigma$. An expanding map of type $\Sigma$ is a map $g$ with the following properties:
	
	\begin{enumerate}
		\item[(1)] the domain of $g$ is a disjoint union $\sqcup_{\mathcal{B}}I(a,b)$, where, for each $(a,b)$, $I(a,b)$ is a compact subinterval of $I(a):=[0,1]\times {a}$;
		
		\item[(2)] for each $(a,b)\in \mathcal{B}$, the restriction of $g$ to $I(a,b)$ is a smooth diffeomorphism onto $I(b)$ satisfying $|Dg(t)|>1$, for all $t$.
	\end{enumerate}
	
	The regular Cantor set associated to $g$ is defined as the maximal invariant set
	
	$$K=\bigcap_{n\ge 0}g^{-n}\left(\sqcup_{\mathcal{B}}I(a,b)\right).$$
	
\end{remark}

\begin{remark}\label{rr2}
	These two definitions are equivalent. On one hand, we may, in the first definition take $I(j):=I_j$, for each $1\le j\le r$, and for each pair $j,k$ such that $\psi(I_j)\supset I_k$, take $I(j,k)=I_j\cap \psi^{-1}(I_k)$. Conversely, in the second definition, we could consider an abstract line containing all intervals $I(a)$ as subintervals, and $\{I(a,b);(a,b)\in \mathcal{B}\}$ as the Markov partition. Moreover we write the data $(\mathcal{A},\mathcal{B},\Sigma,g)$ defining $K$.
\end{remark}

The Markov partition $\mathcal{P}=\{I_1,I_2,...,I_r\}$ and the map $\psi$ define an $r\times r$ matrix $B=(b_{ij})_{i,j=1}^r$ by

$$b_{ij}= \left\{\begin{array}{rc}
1, &\mbox{if}\quad \psi(I_i)\supset I_j\\ 
0, &\mbox{if}\quad \psi(I_i)\cap I_j=\emptyset \end{array}\right.\\$$
which encodes the combinatorial properties of $K$. To give $\mathcal{B}$ is equivalent to give $B$. We put $\Sigma_B=\{\underline{\theta}=(\theta_1,\theta_2,...)\in \{1,2,...,r\}^{\mathbb{N}}; b_{\theta_i,\theta_{i+1}}=1, \ \forall i\ge 1\}$ and $\sigma:\Sigma_B\rightarrow \Sigma_B$ such that $\sigma(\underline{\theta})=\underline{\tilde{\theta}}$, $\tilde{\theta}_i=\theta_{i+1}$ we call the shift map. 
There is a natural homeomorphism between the pairs $(K,\psi)$ and $(\Sigma_B,\sigma)$. For each finite word $(a_1,a_2,...,a_n)$ such that $(a_i,a_{i+1})\in \mathcal{B}, 1\le i\le n-1$, the intersection 
$$I_{\underline{a}}=I_{a_1}\cap \psi^{-1}(I_{a_2})\cap ...\cap \psi^{-(n-1)}(I_{a_n})$$  
is a non-empty interval and by the mean-value theorem there is $x\in I_{\underline{a}}$ such that
$$\mbox{diam}(\psi^{n-1}(I_{\underline{a}}))=|(\psi^{n-1})'(x)|\mbox{diam}(I_{\underline{a}}),$$
but, since $\psi^{n-1}(I_{\underline{a}})\subset I_{a_n}$ we have 
$$\mbox{diam}(I_{\underline{a}})=\frac{\mbox{diam}((\psi^{n-1}(I_{\underline{a}}))}{|(\psi^{n-1})'(x)|}\le \frac{\mbox{diam}(I_{a_n})}{|(\psi^{n-1})'(x)|}$$
which is exponentially small if $n$ is large. Then, the map 
$$\begin{array}{rcl}
h:\Sigma&\rightarrow & K \\
\underline{\theta}& \mapsto& \cap_{n\ge 1}I_{(\theta_1,\theta_2,...,\theta_n)}
\end{array}$$
defines a homeomorphism between $\Sigma_B$ and $K$ such that 
$$h\circ \sigma=\psi \circ h.$$ 

We will denote $\hat{I}_{\underline{a}^1,\underline{a}^2,...,\underline{a}^k}$ the convex hull of intervals $I_{\underline{a}^1}\cup I_{\underline{a}^2}\cup ...\cup I_{\underline{a}^k}$, where $\underline{a}^j=(a^j_1,a^j_2,...,a^j_{n_j})$, $1\le j\le k$, and $\underline{a}^j$ is an \textbf{admissible} word, in the sense that if $\underline{a}^j=(a^j_1,...,a^j_{n_j})$ then $(a_i,a_{i+1})\in \mathcal{B}$  for all $1\le i\le n_j-1$.
We say that a regular Cantor set is \textbf{affine} if $\psi_{|I_j}$ is affine and onto $\hat{I}_{1,r}$, the convex hull of $I_1,...,I_r$. Let $K$ be an affine regular Cantor set. We know that (cf. \cite{PT93}, pag. 71) if $|\psi'_{|I_j}|=\lambda_j$ then the Hausdorff dimension of $K$ is the only number $d>0$ such that
$$\sum_{j=1}^r\lambda_j^{-d}=1.$$
Note that, in the case of affine regular Cantor sets, if we forget any non-extremal interval of the Markov partition $\mathcal{P}$, say $J$, we still have that $\tilde{\mathcal{P}}=\mathcal{P}-\{J\}$ is a Markov partition of another regular Cantor set, 
$$\tilde{K}=\bigcap_{n\ge  0}\psi^{-n}\left(\bigcup_{I\in \tilde{\mathcal{P}}}I\right).$$
The following example give us an idea of how to decrease the Hausdorff dimension of a regular Cantor set. 

\begin{example}Let $(K,\mathcal{P},\psi)$ be an affine regular Cantor set. We suppose that $|\psi'_{|I_j}|=\lambda_j$. We know that the Hausdorff dimension of $K$ is the number $d$ such that
	$$\sum_{j=1}^r \lambda_j^{-d}=1.$$
	So, if we forget, for example, $I_2$, the new regular Cantor set defined by $\mathcal{\tilde{P}}=\{I_1,I_3,...,I_r\}$ 
	has Hausdorff dimension $\tilde{d}$ such that
	$$\sum_{j=1,j\neq 2}^r \lambda_j^{-\tilde{d}}=1.$$
	It follows that $\sum_{j=1}^r \lambda_j^{-\tilde{d}}>1$ and therefore $d>\tilde{d}$, i.e., $HD(K)>HD(\tilde{K}).$
\end{example}

Recall that the pressure of $\psi$ and a potential $\phi$ is given by
\begin{eqnarray}\label{1}
P(\psi,\phi)=\sup\{h_{\mu}(\psi)+\int \phi \ d\mu; \mu \ \mbox{is \ an \ invariant \ measure}\}.
\end{eqnarray}
Moreover, by the Ergodic Decomposition Theorem and the Jacobson's theorem, the last supremum can be taken on the ergodic invariant measure. We say that a measure $m$ is an \textbf{equilibrium state} if the supremum is attained it in (\ref{1}). When $\psi$ is $C^{1+\alpha}$ and the potential $\phi=-s\log|\psi'|$, we know that (cf. \cite{P98}, pag. 203, thm 20.1, (2)) there exists an unique equilibrium measure and it is equivalent to the $d$-dimensional Hausdorff measure, where $d$ is the Hausdorff dimension of Cantor set defined by $\psi$.

In the next lemma, we show that if $K$ is a regular Cantor set and $J$ is an interval of the construction of $K$, i.e., there are $m\in \mathbb{N}$ and $j\in \{1,2,...,r\}$ such that $J\in \psi^{-(m-1)}(I_j)$, the Cantor set $\tilde{K}$ obtained by the omission of $J$ has Hausdorff dimension strictly smaller than the Hausdorff dimension of $K$. Precisely

\begin{lemma}\label{L1}Given $(K,\mathcal{P},\psi)$ a $C^s$-regular Cantor set, $s>1$, we denote $\mathcal{P}^m$ the set of connected components of $\psi^{-(m-1)}(I_j)$, $1\leq j\leq r$. 
	If $J\in \psi^{-(m-1)}(I_j)$ for some natural number $m$ and some $j\leq r$ such that $\mathcal{P}^m-\{J\}$ is a Markov partition of $\psi$, then the Cantor set
	$$\tilde{K}=\bigcap_{n\geq 0}\psi^{-nm}\left(\bigcup_{I\in \mathcal{P}^m, I\neq J}I\right )$$ 
	satisfies $HD(\tilde{K})<HD(K)$.
\end{lemma}

\begin{proof}
	Note that we can write $(K,\mathcal{P},\psi)=(K,\mathcal{P}^m, \psi^m)$, in the sense that
	$$K=\bigcap_{n\geq 0}\psi^{-n}\left(\bigcup_{I\in \mathcal{P}}I\right )=\bigcap_{n\geq 0}\psi^{-mn} \left(\bigcup_{I\in \mathcal{P}^m}I\right ).$$ So, to show the lemma is equivalent to show that if we forget some interval of $\mathcal{P}$ then the Hausdorff dimension decreases. 
	
	Let $d=HD(K)$ and $m_d$ be the Hausdorff dimension of $K$ and the $d$-dimensional Hausdorff measure, respectively. We know that $m_d(K)>0$ and there exists $c>0$ such that, for all $x\in K$ and $0<r\leq 1$ (cf. \cite{PT93}, pag. 72, proposition 3),
	$$c^{-1}\leq \frac{m_d(B_r(x)\cap K)}{r^d}\leq c.$$
	Moreover, if $\mu$ denotes the unique equilibrium measure corresponding to the H\"older continuous potential $-s \log |\psi'|$, we have that $m_d$ is equivalent to $\mu$ (cf. \cite{P98}, pag. 203). 
	
	By uniqueness, $\mu$ is an ergodic invariant measure for $\psi$. Consider for $x\in K$ 
	$$\tau(I_l\cap K;x)=\lim_{n\to \infty}\frac{\#\{0\leq j\leq n-1; \psi^j(x)\in I_l\cap K\}}{n}.$$  
	From the Birkhoff's Ergodic Theorem, $\tau(I_l\cap K;x)=\mu(I_l\cap K)$ for $\mu$-almost every $x\in K$. Take $y\in I_l\cap K$ and $\tilde{I}\subset I_l$ any interval centred in $y$. Note that $\mu(I_l\cap K)>0$  because $$\displaystyle m_d(I_l\cap K)\geq m_d(\tilde{I}\cap K)\geq \frac{c^{-1}}{2^d}|\tilde{I}|^{d}>0.$$ This implies that the set of points which never visit $I_l$ has measure zero. Note that this set contains  
	$$\tilde{K}=\bigcap_{n\geq 0}\psi^{-n}(I_1\cup ...\cup I_{l-1}\cup I_{l+1}\cup...  \cup I_r).$$ Since $\mu$ is equivalent to $m_d$, we have $m_d(\tilde{K})=0.$ On the other hand, if $d=HD(\tilde{K})$ we would have $m_d(\tilde{K})>0$. So $HD(\tilde{K})<d=HD(K).$
	
\end{proof}

%If $\underline{a}\in \mathcal{B}^{\infty}$ is in $\mathcal{B}_n$ we write $|\underline{a}|=n$ the length of $\underline{a}$. For us $\mathcal{B}_1=\mathcal{A}$ and $\mathcal{B}_2=\mathcal{B}$.
%Our next result is technical and we need some definitions.

%\begin{definition}
%	The \textbf{final} (resp. \textbf{initial}) \textbf{erase map} $E_f:\mathcal{B}^{\infty}\rightarrow \mathcal{B}^{\infty}$ is defined by $E_f({\underline{a}})=(a_1,a_2,...,a_{n-1})$ (resp.  $E_c({\underline{a}})=(a_2,...,a_n)$), if $\underline{a}=(a_1,a_2,...,a_n).$ We set $E_f(( \ ))=( \ )=E_c(( \ ))$ where $( \ )$ is the \textbf{empty word}.
%\end{definition}

%\begin{definition}
	%Given a word $\underline{a}\in \mathcal{B}_n$ we call another word $\underline{b}\in \mathcal{B}_n$ a \textbf{sister} of $\underline{a}$ and we denote $\underline{a}\sim\underline{b}$, if $E_f(\underline{a})=E_f(\underline{b})$. Moreover, two intervals $I_{\underline{a}}$ and $I_{\underline{b}}$ are \textbf{brothers} and we denote $I_{\underline{a}}\sim I_{\underline{b}}$,  iff $\underline{a}\sim\underline{b}$. We call the interval $I_{\underline{a}}$ of \textbf{child}  of $I_{E_f(\underline{a})}$.
%\end{definition}

%\begin{remark}\label{r1}
%	It follows from the above definitions that $\sim$ is an equivalence relation on $\mathcal{B}_n$ and $\mathcal{P}^n=\cup_{1\le j\le r}\psi^{-(n-1)}(I_j)$. We note that the convex hull of all brothers of $I_{\underline{a}}$ is precisely $I_{E_f(\underline{a})}$. Moreover, if $I_{\underline{a}}\sim I_{\underline{b}}$ then $I_{E_c({\underline{a}})}\sim I_{E_c({\underline{b}})}$.
%\end{remark}

\begin{definition}
	For $a\in \mathcal{A}$, consider $\mathcal{B}(a)=\{b\in \mathcal{A}; (a,b)\in \mathcal{B}\}$. The \textbf{degree} of $a$ is defined by $\deg(a)=\# \mathcal{B}(a)$, as the number of elements $b\in \mathcal{A}$ such that $(a,b)\in \mathcal{B}$.
	In other words, $\deg(a)$ represents the number of intervals that are visited by the interval $I_a$. In particular, $\psi(I_a)$ is the convex hull of $\deg(a)$ intervals. In a similar way, if $\underline{a}=(a_1,...,a_n)\in \mathcal{B}_n$ we define $\deg(\underline{a}):=\deg(a_n)$.
\end{definition}

\begin{definition}\label{i1}
	We say that a Markov partition $\mathcal{P}$ has the interior property (IP) if there exists an interval $I\in \mathcal{P}$ such that for every $J\in \mathcal{P}$ with the following property: 
	\begin{equation}\label{eq21}
	\mbox{If} \  \psi(J)\supset I \ \mbox{then} \ \inte(\psi(J))\supset I,
	\end{equation} where $\inte{(A)}$ denotes the interior of a set $A$. 
\end{definition}

\begin{example}
	Let $\mathcal{P}=\{I_1,I_2,I_3\}$  be a Markov partition of a regular Cantor set $K$, defined by an expanding map $\psi$ satisfying $\psi(I_1)=\hat{I}_{1,2}$, $\psi(I_2)=\hat{I}_{2,3}$ and $\psi(I_3)=I_1$. Thus, $\mathcal{A}=\{1,2,3\}$ and $\mathcal{B}=\{(1,1),(1,2),(2,2),(2,3),(3,1)\}$. Note that $\psi(I_1)\supset I_2$, but $\inte(\psi(I_1))\cap \inte(I_2)=\inte(I_2)$.
	Then $\mathcal{P}$ does not has (IP).%However, consider $\tilde{\mathcal{P}}=\{I_{111},I_{112},I_{12},I_2,I_3\}$. Note that $\tilde{\mathcal{P}}$ is a Markov partition. In fact, $\psi(I_{111})=I_{11}=\hat{I}_{111,112}$, $\psi(I_{112})=I_{12}$, $\psi(I_{12})=\psi(I_2)$, $\psi(I_2)=\hat{I}_{2,3}$ and $\psi(I_{3})=I_1=\hat{I}_{111,112,12}$. Note that $\psi(I_{111})\supset I_{112}$ but $\inte(\psi(I_{111}))\nsupseteq I_{112}$. Moreover, there is no element $I\in \tilde{\mathcal{P}}$ such that 
	%$$\psi(J)\supset I\Rightarrow \inte(\psi(J))\supset I.$$
	%Then $\tilde{\mathcal{P}}$ does not has (IP).
\end{example}

In the above example we note that $\deg(a)\le 2$ for all $a\in \mathcal{A}$. We shall see below when there is a word $a$ with $\deg(a)=3$, easily we can to obtain a refinement $\mathcal{\tilde{P}}$ of $\mathcal{P}$ such that $\mathcal{\tilde{P}}$ has $(IP)$.

\begin{example}
	Let $\mathcal{P}=\{I_1,I_2,I_3\}$ be a Markov partition of a regular Cantor set $K$, defined by an expanding map $\psi$ satisfying $\psi(I_1)=\hat{I}_{1,2,3}$, $\psi(I_2)=\hat{I}_{2,3}$ and $\psi(I_3)=I_1$, $|I_3|<|I_1|$,. So $\mathcal{A}=\{1,2,3\}$ and $\mathcal{B}=\{(1,1),(1,2),(1,3),(2,2),(2,3),(3,1)\}$. Now, we consider $\tilde{\mathcal{P}}=\{I_{11},I_{12},I_{13},I_2,I_3\}$. Hence, in a similar way as the above example, we have that $\mathcal{\tilde{P}}$ is a Markov partition. Observe that $\deg(1)=3$ and $$I_{12}\subset \inte(I_1).$$
	Since for every $J\in \tilde{\mathcal{P}}$ we have $\psi(J)=I_a$ for some $a\in \mathcal{A}$, it follows that $$\psi(J)\supset I_{12}\Rightarrow \inte(\psi(J))\supset I_{12}.$$ 
	In this present case we have that $\tilde{\mathcal{P}}$ has $(IP)$.
\end{example}

\begin{lemma}\label{L2}
	If $(K,\mathcal{P},\psi)$ is a $C^s$-regular Cantor set such that either there is $j$ with $\deg(j)\ge 3$ or $\deg(j)=2$ for every $j$ then there exists a Markov partition $\tilde{\mathcal{P}}\prec \mathcal{P}$ for $\psi$ determining $K$ such that $\tilde{\mathcal{P}}$ has $(IP)$.
	
\end{lemma}

\begin{proof}
	We begin supposing that there is $j\in \mathcal{A}$ such that $\mbox{deg}(j)\ge 3$. Thus, there is $l(j)$ such that $\mbox{int}(I_j)\supset I_{jl(j)}$.
	\begin{center}
		
		\begin{figure}[!h]
			\includegraphics[trim=0.5cm 8.5cm 2cm 1.3cm]{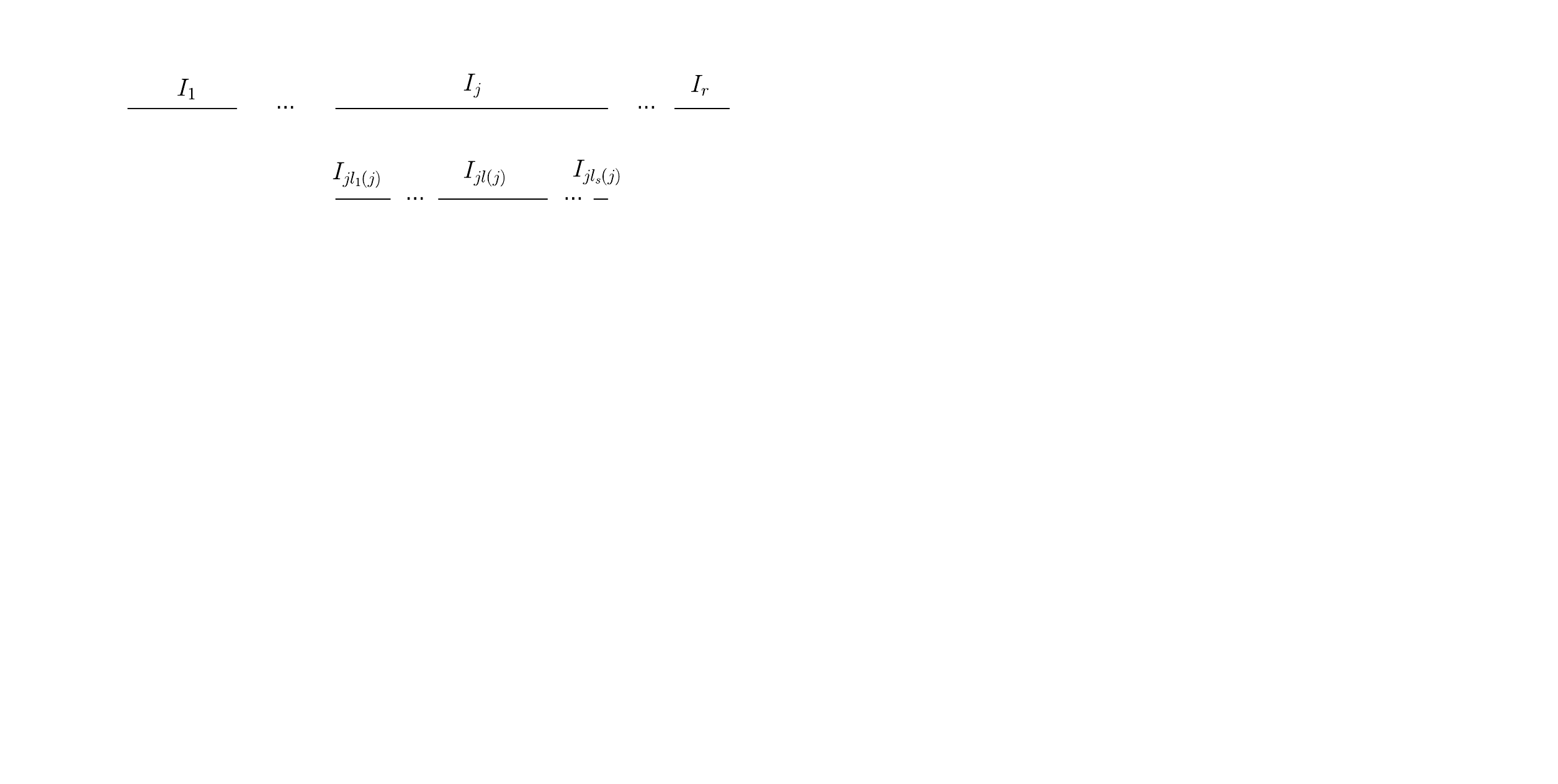}
			\caption{arrangement of intervals}
			\label{fig.1}
		\end{figure}
		
	\end{center}
	Consider $\tilde{\mathcal{P}}=(\mathcal{P}-\{I_j\})\cup (\cup_{b\in \mathcal{B}(j)}I_{jb}).$ 
	It is easy to see that $\tilde{\mathcal{P}}$ is a Markov partition to $\psi$. If $b\neq j$ then $\psi(I_{jb})=I_b$ and $I_b\in \tilde{\mathcal{P}}$, because $b\neq j$ . If $j\in \mathcal{B}(j)=\{l_1,...,l_s\}$ then $\psi(I_{jj})=I_j=\hat{I}_{jl_1,...jl_s}$. For every $I_a\in \tilde{\mathcal{P}}$ with $I_a\in \mathcal{P}$, the convex hull does not change, so $\psi(I_a)$ is the convex hull of elements in $\tilde{\mathcal{P}}$ because if $\psi(I_a)\supset I_j$ then we use that $I_{j}=\hat{I}_{jl_1,...,jl_s}$.   Moreover, if $\psi(I_a)\supset I_{jl(j)}$ then $\mbox{int}(\psi(I_a))\supset I_{jl(j)},$ because $\psi(I_a)\cap I_{jl(j)}\neq \emptyset$ implies that $\psi(I_a)\cap I_j\neq \emptyset$ and so $\psi(I_a)\supset I_j$. But $\inte(I_j)\supset I_{jl(j)}$, thus $\tilde{\mathcal{P}}$ is the Markov partition in the statement with $I=I_{jl(j)}$.
	
	Next we suppose that $\deg(j)=2$ for every $j\le \# \mathcal{P}$. In particular, $\deg(1)=2$ and there are $a$ and $b$ such that $\hat{I}_{1\ell_1,1\ell_2}=I_1$. There is no loss of generality suppose that $\ell_2\neq 1$. Since $\deg(\ell_2)=2$ there are $a,b$ such that $\hat{I}_{1\ell_2a,1\ell_2b}=I_{1\ell_2}$ and there are $c,d$ such that $\hat{I}_{1\ell_2bc,1\ell_2bd}=I_{1\ell_2b}$, see figure \ref{fig2}. 	\begin{center}
	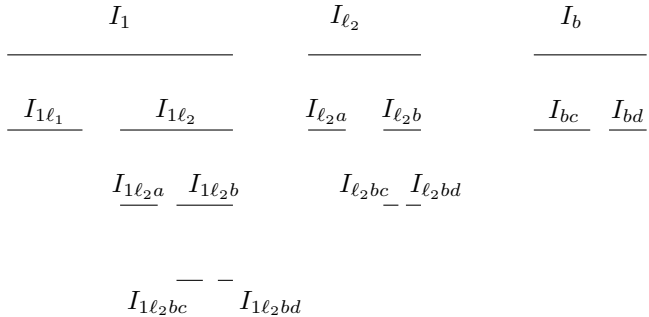
\begin{figure}\label{fig2}
	\begin{tikzpicture}
	\draw (-3,0) -- (0,0);
	\draw (-1.5,0.5) node {$I_1$};
	\draw (-3,-1.0) -- (-2,-1.0);
	\draw (-2.5,-0.75) node {$I_{1\ell_1}$};
	\draw ( -1.5,-1.0) -- (0,-1.0);
	\draw (-0.75,-0.75) node {$I_{1\ell_2}$};
	\draw (-1.5, -2) -- (-1,-2);
	\draw (-1.25,-1.75) node {$I_{1\ell_2a}$};
	\draw (-0.75,-2) -- (0,-2);
	\draw (-0.25,-1.75) node {$I_{1\ell_2b}$};
	\draw (-0.75,-3) -- (-0.4,-3);
	\draw (-1,-3.3) node {$I_{1\ell_2bc}$};
	\draw (-0.2,-3) -- (0,-3);
	\draw (0.5,-3.3) node {$I_{1\ell_2bd}$};
	\draw (1,0) -- ( 2.5,0);
	\draw (1.5,0.5) node {$I_{\ell_2}$};
	\draw (1,-1)--(1.5,-1);
	\draw (1.25,-0.75) node {$I_{\ell_2a}$};
	\draw (2,-1) -- (2.5,-1);
	\draw (2.25,-0.75) node {$I_{\ell_2b}$};
	\draw (2,-2)--(2.2,-2);
	\draw (1.75,-1.75) node {$I_{\ell_2bc}$};
	\draw (2.3,-2) -- (2.5,-2);
	\draw (2.7,-1.75) node {$I_{\ell_2bd}$};
	\draw (4,0) -- (5.5,0);
	\draw (4.5, 0.5) node {${I_b}$}; 
	\draw (4,-1) -- (4.75,-1);
	\draw (4.4,-0.75) node {$I_{bc}$};
	\draw (5,-1) -- (5.5, -1);
	\draw (5.25,-0.75) node {$I_{bd}$};
	\end{tikzpicture}
	\caption{Configuration of the intervals in the preserving orientation case}
	\end{figure}
	\end{center}
	Consider 
	$$\tilde{\mathcal{P}}=(\mathcal{P}\setminus \{I_1,I_{\ell_2},I_b\})\cup \{I_{1\ell_2bc},I_{1\ell_2bd},I_{1\ell_2a},I_{\ell_2bc},I_{\ell_2bd},I_{1\ell_1},I_{\ell_2a},I_{bc},I_{bd}\}.$$
Then $\tilde{\mathcal{P}}$ is the Markov partition of the statement and $I:=I_{1\ell_2bc}$ is the interval such that if $\psi(J)\supset I$ then $\inte(\psi(J))\supset I$.	
\end{proof}

\begin{remark}
Thanks the lemma \ref{L2} if we forget $I$, that is, if we consider $\mathcal{P}_0=\mathcal{P}-\{I\}$, then $\mathcal{P}_0$ is still a Markov partition of $\psi$, and thanks the lemma (\ref{L1}) the Cantor set 
$$K_0=\bigcap_{n\ge 0}\psi^{-n}\left(\bigcup_{J\in \mathcal{P}_0}J\right)$$
has Hausdorff dimension strictly smaller than $K$. 

%Moreover, we can suppose that our Markov partitions have the property of above lemma. We will suppose this one.
\end{remark}

\begin{remark}\label{re 2}
Note that the property of the above lemma is a $C^0$-open property. We use this one in the next section. 
\end{remark}

Next, we define when two regular Cantor sets, $K$ and $K'$ are $C^s$-close.

\begin{definition}
Given a $C^s$, $s\ge 1$, regular Cantor set, $K$, with a Markov partition $\mathcal{P}$, we say that another $C^s$ regular Cantor set, $K'$, with a Markov partition $\mathcal{P}'$, is $C^s$ close to $K$ if $\#\mathcal{P}=\#\mathcal{P}'$, the extremal points of any element $P_i\in \mathcal{P}$ is close to the extremal points of $P_i'\in \mathcal{P}'$ and maps $\psi$ and $\psi'$ are close in the $C^s$ topology.
\end{definition}

\begin{example}
Consider the intervals $I_1,I_2,I_3, I_4$ such that $|I_1|<|I_3|$, $|I_2|<|I_3|<|i_4|$. Let $\psi: I_1\cup I_2 \cup I_3\cup I_4\rightarrow \hat{I}_{1,4}$ such that $\psi(I_1)=I_3$, $\psi(I_2)=I_3$, $\psi(I_3)=i_4$ and $\psi(I_4)=\hat{I}_{1,4}$ in the affine and increasing way. Therefore $\psi'>1$. We have $\deg(1)=\deg(2)=\deg(3)=1$ and $\deg(4)=4$. We define 
$$T(\psi):I_1\cup I_2 \cup I_3\cup I_4\rightarrow \hat{I}_{1,4}$$ by
	\begin{equation}\label{eq1.exampleT}
T(\psi)|_{I_1}=\psi^3=\psi|_{I_4}\circ \psi|_{I_3}\circ \psi|_{I_1}
	\end{equation}
	\begin{equation}\label{eq2.exampleT}
T(\psi)|_{I_2}=\psi^3=\psi|_{I_4}\circ \psi|_{I_3}\circ \psi|_{I_2}
	\end{equation}
	\begin{equation}\label{eq.exampleT}
T(\psi)|_{I_1}=\psi^3=\psi|_{I_4}\circ \psi|_{I_3}\circ \psi|_{I_1}
	\end{equation}
	\begin{equation}\label{eq.exampleT}
T(\psi)|_{I_3}=\psi^2=\psi|_{I_4}\circ \psi|_{I_3}
	\end{equation}
	\begin{equation}\label{eq.exampleT}
T(\psi)|_{I_4}=\psi=\psi|_{I_4}.
	\end{equation}
Note that by definition $T(\psi)|_{I_j}(I_j)=\hat{I}_{1,4}$ and $\deg(j)=4$ for every $j\in \{1,2,3,4\}$. We also have $T(\psi)'>1$.  Consider 
$$K=\bigcap_{n\ge 0}\psi^{-n}(I_1\cup I_2\cup I_3\cup I_4) \ \ \mbox{and} \ \ \tilde{K}=\bigcap_{n\ge 0}T(\psi)^{-n}(I_1\cup I_2\cup I_3\cup I_4).$$
\begin{claim}
$K=\tilde{K}$.
\end{claim}
Note that $x\in K$ if and only if $\psi^n(x)\in I_1\cup I_2\cup I_3\cup I_4$. While $x\in \tilde{K}$ if and only if $T(\psi)^n\in I_1\cup I_2\cup I_3\cup I_4$. Then, it follows from the definition of $T(\psi)$ that $K\subset \tilde{K}.$ 

On the other hand, take $x\in \tilde{K}$. 

\begin{enumerate}

\item[(i)] If $x\in I_1\cup I_2$ then for $i=1$ or $2$ we have 
	$$T(\psi)(x)=\psi|_{I_4}\circ \psi|_{I_3}\circ \psi|_{I_i}(x)\in I_1\cup I_2\cup I_3\cup I_4.$$
	That is, 
	$$\psi(x)=\psi|_{I_i}(x)\in (\psi|_{I_3})^{-1}\circ (\psi|_{I_4})^{-1}(I_1\cup I_2\cup I_3\cup i_4)=(\psi|_{I_3})^{-1}(I_4)=I_3.$$
	 In particular, $\psi^2(x)=\psi|_{I_3}\circ \psi|_{I_i}(x)\in I_4$.  
	 By definition, 
	 $$\psi^3(x)=T(\psi)(x)\in I_1\cup I_2\cup I_3\cup I_4.$$
	 Following the above argument, $\psi^n(x)\in I_1\cup I_2\cup I_3\cup I_4$ for all $n\ge 0.$
\item[(ii)] If $x\in I_3$ then $T(\psi)(x)=\psi^2(x)=\psi|_{I_4}\circ \psi|_{I_3}(x)\in I_1\cup I_2\cup I_3\cup I_4.$ In particular, 
$$\psi(x)=\psi|_{I_3}(x)\in (\psi|_{I_4})^{-1}(I_1\cup I_2\cup I_3\cup i_4)=I_4.$$

\item[(iii)] If $x\in I_4$ then $T(\psi)(x)=\psi(x)\in I_1\cup I_2\cup I_3\cup I_4$.
\end{enumerate}
By (i), (ii) and (iii), $\psi^n(x)\in I_1\cup I_2\cup I_3\cup I_4$ for all $n\ge 0$. That is, $x\in K$. This proves the claim.

Next, we take $\tilde{\phi}:\tilde{I_1}\cup \tilde{I_2}\cup \tilde{I_3}\cup \tilde{I_4}\rightarrow \hat{\tilde{I}}_{1,4}$ close to $T(\psi)$ in the $C^s$ topology and such that the extremal points of $I_j$ are close to the extremal point of $I_j$. In particular, $K\sim K_{\tilde{\phi}}$ in the $C^s$ topology. Let $J_i=I_i\cap \tilde{I}_i$. We want to find $\phi\sim \psi$ in the $C^s$ topology such that $T(\phi)=\tilde{\phi}$. An algorithm to do that is the following. First of all, we know that $T(\psi)|_{I_4}=\psi|_{I_4}$. By hypotesis, $\psi|_{J_4}\sim \tilde{\phi}|_{J_4}$. Then, we put, $\phi|_{\tilde{I}_4}=\tilde{\phi}|_{\tilde{I}_4}$. In the following step we look for $T(\psi)|_{I_{3}}=\psi|_{I_4}\circ \psi|_{I_3}$. Since $T(\psi)|_{J_3}\sim \tilde{\phi}|_{J_3}$, that is, $\psi|_{J_4}\circ \psi|_{J_3}\sim \tilde{\phi}|_{J_3}$, in particular, $\psi|_{J_3}\sim (\tilde{\phi}|_{J_4})^{-1} \circ \tilde{\phi}|_{J_3}$ we put $\phi|_{\tilde{I}_3}=(\tilde{\phi}|_{\tilde{I_4}})^{-1}\circ \tilde{\phi}|_{\tilde{I_3}}.$ The other two step are clear. For $i=1,2$ we begin with $T(\psi)|_{I_i}=\psi|_{I_4}\circ \psi|_{I_3}\circ \psi|_{I_i}\sim \tilde{\phi}|_{J_i}$. Then, $\psi|_{I_3}\circ \psi|_{I_i}\sim (\tilde{\phi}|_{\tilde{I_4}})^{-1}\circ \tilde{\phi}|_{J_i}$. But, we know that $\psi|_{J_3}\sim (\tilde{\phi}|_{J_4})^{-1} \circ \tilde{\phi}|_{J_3}$. This implies that 
$$(\tilde{\phi}|_{J_4})^{-1} \circ \tilde{\phi}|_{J_3}\circ \psi|_{I_i}\sim \psi|_{I_3}\circ \psi|_{I_i}\sim (\tilde{\phi}|_{\tilde{I_4}})^{-1}\circ \tilde{\phi}|_{J_i}.$$ That is, $\phi|_{\tilde{I}_i}:=(\tilde{\phi}|_{\tilde{I}_3})^{-1}\circ \tilde{\phi}|_{\tilde{I}_{i}}\sim \psi|_{I_i}$. By construction we have that
$T(\phi)=\tilde{\phi}$.

\end{example}

Following the above example we can prove the Lemma below. Since the proof follows the same ideas as the above example but more longer, we omit here.

\begin{lemma}\label{Lema de pert}
Fix $\psi:I_1\cup ... \cup I_r\rightarrow \hat{I}_{1,r}$ such that $\psi$ defines a $C^s$ regular Cantor set $K$ such that $\deg(j)\le 2$. For $j\le r$ consider 
$$m_j=\min\{n; \psi^n(I_j) \mbox{cover two intervals of} \ \mathcal{P} \}.$$
 We define $T(\psi)$ by $T(\psi)|I_j=\psi^{m_j}$. Then $K_{\psi}=K_{T(\psi)}$. Moreover, If $\tilde{\phi}$ is $C^{\tilde{s}}$ close to $T(\psi)$ then there is $\phi$ such that $\phi$ is $C^{\tilde{s}}$ close to $\psi$ and $T(\phi)=\tilde{\phi}$.

\end{lemma}

The Lemmas \ref{L2} and \ref{Lema de pert} we can assume that our dynamics has a Markov partition with the Interior Property.
%\begin{proof} Observe that by the definition of regular Cantor set there is $M\in \mathbb{N}$ such that $m_j\leq M$ for every $j\in \{1,2,...,r\}$.
%We note that given $j$ there is a finite sequence of intervals $I_{a(j)}, I_{a^2(j)},...,I_{a^{m_j-1}(j)}$ such that $\psi$ just apply $I_{a^i(j)}$ diffeomorphically on $I_{a^{i+1}(j)}$ and $\deg(a^{m_j-1}(j))\geq 2$.

%First we say that $j$ is in the $n$-th level if $\psi^{k}(I_j)=I_{i_k}$ for $k\le n$ but $\psi^n(I_j)$ cover two intervals of $\mathcal{P}$. We write
%$C_n=\{j; j \ \mbox{is in the} \ n-\mbox{th level}\}$.
%\end{proof}

We remember that if we have a regular Cantor set $K$ with a Markov partition $\mathcal{P}=\{P_1,...,P_r\}$ we can define $\mathcal{P}^n$ as being the connected component of $\psi^{-(n-1)}(I_j)$, $1\le j\le r$. Moreover, if we write
$$\lambda_{n,P}=\inf_{x\in P}|(\psi^n)'(x)|, \ P\in \mathcal{P}^n$$
and define $\beta_n$ implicitly by
$$\sum_{P\in \mathcal{P}^n}\lambda^{-\beta_n}_{n,P}=1$$
then $\beta_n\ge HD(K)$ for all $n\in \mathbb{N}$ and
\begin{equation}\label{hd}
\lim_{n\to \infty}\beta_n=HD(K).
\end{equation}
see \cite{PT93} pp 68-71.
\begin{proposition}\label{P1}
Given any $C^{\omega}$-regular Cantor set $(K,\mathcal{P},\psi)$ there exist $C^{k}$-arbitrarily close to it, $k\in \mathbb{N}$, another $C^{\omega}$-regular Cantor set $(\tilde{K},\tilde{\mathcal{P}},\tilde{\psi})$ with

$$HD(\tilde{K})\neq HD(K).$$
\end{proposition}

\begin{proof}
By Lemma \ref{L2} we can suppose that $\mathcal{P}$ has the $(IP)$. Let $I_l=[a_l,b_l]$ be the set of $\mathcal{P}$ as the above lemma. For $\lambda\in (0,1)$ let us define, $b^{\lambda}_l=(1-\lambda)a_l+\lambda b_l$, $I^{\lambda}_l=[a_l,b^{\lambda}_l]$, $h_{\lambda}:I^{\lambda}_l\rightarrow I_l$, by $\displaystyle h_{\lambda}(x)=\frac{x-(1-\lambda)a_l}{\lambda}$ and $\psi_{\lambda}$ in the following way:

$$\psi_{\lambda}(x)=\left\{\begin{array}{rc}
\psi(h_{\lambda}(x)),&\mbox{if}\quad x\in I^{\lambda}_l\\ 
\psi(x), &\mbox{otherwise}\end{array}\right.$$

Note that $\psi_{\lambda}(a_l)=\psi(a_l)$ and $\psi_{\lambda}(b^{\lambda}_l)=\psi(b_l)$ and therefore $\psi_{\lambda}(I^{\lambda}_l)=\psi(I_l)$. It follows that $\psi_{\lambda}(J)=\psi(J)$ for every $J\in \mathcal{P}.$ Moreover, $|\psi'_{\lambda}|>\frac{1}{\lambda}|\psi'|$ on $I^{\lambda}_l$. In particular, $\inf |\psi'_{\lambda}|\to +\infty$ when $\lambda\to 0$. Consider $\mathcal{P}^{\lambda}=(\mathcal{P}-\{I_l\})\cup \{I^{\lambda}_l\}.$ We claim that $\mathcal{P}^{\lambda}$ is a Markov partition to $\psi_{\lambda}.$ We need only to prove that $\psi_{\lambda}(I_s)$ is the convex hull of a union of $I_t$'s. Let $\mathcal{S}$ be the set of index such that if $s\in \mathcal{S}$ then $\psi(I_s)\supset I_l$. If $s\notin \mathcal{S}$ then $\psi_{\lambda}(I_s)=\psi(I_s)$ is the convex hull of intervals distinct from $I_l^{\lambda}$. If $s\in \mathcal{S}$ then 

$$\psi_{\lambda}(I_s)=\psi(I_s)=\hat{I},$$
where $I=I_{j_1}\cup ...\cup I_l\cup... \cup I_{j_p}$ and $\hat{I}$ is the convex hull of $I$. If we set $I_a<I_b$ iff for every $x\in I_a$ and for every $y\in I_b$ we have $x<y$, so $I_l$ cannot be an extremal, because it is contained in the interior of any $\psi(J)$ that contains it. So  
$$\psi_{\lambda}(I_s)=\psi(I_s)=\hat{I}^{\lambda},$$
where $I^{\lambda}=I_{j_1}\cup ...\cup I^{\lambda}_l\cup... \cup I_{j_p}$. Thus, for any $I_s\in \mathcal{P}^{\lambda}$ we have  that $\psi_{\lambda}(I_s)$ is the convex hull of intervals in $\mathcal{P}^{\lambda}.$ We write $\mathcal{P}^{\lambda}=\{I_1^{\lambda},...,I_l^{\lambda},...,I_r^{\lambda}\}$, where $I_j^{\lambda}=I_j$ if $j\neq l$ and $I_l^{\lambda}=[a_l,b^{\lambda}_l].$  We showed above that $\psi_{\lambda}(I^{\lambda}_j)=\psi(I_j)$ for $j\in \{1,2,...,r\}$. Thus, if $k\neq l$ we have
$$\psi_{\lambda}(I^{\lambda}_j)\supset I^{\lambda}_k \Leftrightarrow \psi(I_j)\supset I_k.$$
If $k=l$, $\psi(I_j)\supset I_l$ implies $\psi_{\lambda}(I_j^{\lambda})\supset I_l\supset I_l^{\lambda}.$ Conversely, $\psi_{\lambda}(I^{\lambda}_j)\supset I^{\lambda}_l$ implies that $\psi(I_j)\supset I_l^{\lambda}$ and then $\psi(I_j)\cap I_l\neq \emptyset$, so $\psi(I_j)\supset I_l$. Thus, if we write
$$K_{\lambda}=\bigcap_{-n\ge 0}\psi^n_{\lambda}(I_1\cup... \cup I^{\lambda}_l\cup ... I_r),$$
$K_{\lambda}$ have the same combinatorics for $\lambda\in (0,1].$

Let $\mathcal{B}_n=\{\underline{a}=(a_0,a_1,...,a_{n-1});(a_i,a_{i+1})\in \mathcal{B}, \forall i\in\{0,1,...,n-2\}\}$ be the set of admissible words of length $n$. Given $0\le j\le n-1$ and $k\in \mathcal{A}$ we denote $\mathcal{B}_n(j,k)=\{\underline{a}\in \mathcal{B}_n; a_j=k\}$ and $\mathcal{B}_n(k)=\cup_{0\le j\le n-1}\mathcal{B}_n(j,k)$. We observe that $\mathcal{B}_n(l)^c$ is the set of admissible words of length $n$ associated to $K_0$, where $K_0$ is the regular Cantor set obtained by forgetting $I_l$. Suppose that $d_{\lambda}\neq d_{\mu}$ for some $\lambda\neq \mu$. Since the map $(\lambda,x)\mapsto \psi_{\lambda}(x)$ is analytic, by Ruelle (cf. \cite{R-82}, Corollary 3), the map $\lambda\mapsto HD(K_{\lambda})$ is also analytic. Since analytic functions cannot be locally constant we have that arbitrarily close to $1$ there exists $K_{\lambda}$ such that $HD(K_{\lambda})\neq HD(K).$ Suppose, by contradiction that $\lambda\mapsto d_{\lambda}$ is a constant map, we set $d_{\lambda}=a$, $\lambda\in (0,1]$.

If $\underline{b}=(b_0,...,b_{n-1})\in \mathcal{B}_n(j,l)$ and $x\in I_{\underline{b}}$, from the chain rule, we have
$$|(\psi^n_{\lambda})'_{|I_{\underline{b}}}(x)|=|(\psi^{n-j-1}_{\lambda})'(\psi_{\lambda}^{j+1}(x))||(\psi_{\lambda})'(\psi_{\lambda}^j(x))||(\psi^j_{\lambda})'(x)|.$$ 
Since $\psi^j_{\lambda}(I_{\underline{b}})\subset I^{\lambda}_{b_j...b_{n-1}}\subset I^{\lambda}_l$ if $j<n-1$, we have 

\begin{eqnarray}\label{2.2}
\lambda_{n,\underline{b}}\ge \inf |(\psi_{\lambda}^{n-j-1})'_{|\psi_{\lambda}^{j+1}(I_{\underline{b}})}|\inf |(\psi'_{\lambda})_{|I^{\lambda}_l}|\inf |(\psi_{\lambda}^j)'_{|I^{\lambda}_{\underline{b}}}|.
\end{eqnarray}
The same holds if $j=n-1$, without the first factor, here $\lambda_{n,\underline{b}}=\inf|(\psi^n_{\lambda})'_{|I_{\underline{b}}}|$. From (\ref{2.2}) we have that $\lambda^{-d}_{n,\underline{b}}\to 0$ if $\lambda\to 0$ for every $d>0$ and $n\in \mathbb{N}$. Since $d_{\lambda}=a$ and $d_{\lambda}>d_0>0$ ( Note that $d_0>0$ because $K_0$ is a regular Cantor set ) we have $a>0$. Let us define $\beta_n^{\lambda}$ as the only real number such that 

\begin{equation}\label{2.3}
1=\sum_{\underline{b}\in \mathcal{B}_n}\lambda^{-\beta_n^{\lambda}}_{n,\underline{b}}.
\end{equation}
It follows from (\ref{2.3}) and $a=d_{\lambda}\le \beta_n^{\lambda}$ that 
\begin{equation}\label{2.4}
1\le \sum_{\underline{b}\in \mathcal{B}_n}\lambda^{-a}_{n,\underline{b}}
\end{equation}
Since $\mathcal{B}_n=\mathcal{B}_n(l)\cup \mathcal{B}_n(l)^{c}$, (\ref{2.4}) can be written by
\begin{equation}\label{2.5}
1\le \sum_{\underline{b}\in \mathcal{B}_n(l)}\lambda^{-a}_{n,\underline{b}}+\sum_{\underline{b}\in \mathcal{B}_n(l)^c}\lambda^{-a}_{n,\underline{b}}.
\end{equation}
Thus, since $\lambda_{n,\underline{b}}^{-a}\to 0$ when $\lambda\to 0$ we have that the first parcel in the sum in (\ref{2.5}) goes to zero when $\lambda\to 0$. This implies that
\begin{equation}\label{2.6}
\sum_{\underline{b}\in \mathcal{B}_n(l)^c}\lambda^{-a}_{n,\underline{b}}\ge 1.
\end{equation}
If we define $\beta_n^0$ in the same way as $\beta_n^{\lambda}$, we see that $\beta^0_n\ge a$. However, $\mathcal{B}_n(l)^c$ is the set of admissible words of length $n$ associate to $K_0$ and then we have

$$d_0=\lim_{n\to \infty}\beta_n^0\ge a=d_{\lambda}>d_0,$$
contradiction. Thus, the map $\lambda\mapsto d_{\lambda}$ is not a constant map. Then there is $\lambda$ arbitrarily close to $1$ such that $HD(K_{\lambda})\neq HD(K)$. Moreover, given $k\in \mathbb{N}$ for every $\lambda$ sufficiently close to $0$ we have that $K_{\lambda}$ is $C^k$-arbitrarily close to $K$. Therefore $HD(K_{\lambda})\neq HD(K)$ for some $\lambda$ sufficiently close to $0$. 
\end{proof}

\begin{corollary}Given $a\in (0,1)$, the set $\mathcal{U}=\{K\in \mathcal{K}; HD(K)\neq a\}$ of regular Cantor sets is an open and dense set.
\end{corollary}
\begin{proof}
Since the map $K\mapsto HD(K)$ is a continuous function, $\mathcal{U}$ is an open set. Since the analytic Cantor sets are dense, we can suppose that we have an analytic Cantor set $K$ such that $HD(K)=a$. By the proposition \ref{P1} we can find another analytic Cantor set with $HD(K)\neq a$. This conclude the proof. 
\end{proof}

\section{Non-stability of the Hausdorff dimension of horseshoes}

Let $\varphi:M\rightarrow M$ be a $C^2$-diffeomorphism of a surface with a horseshoe $\Lambda$. Next, following \cite{PT93} let us fix a geometrical Markov partition $\{P_a\}_{a\in \mathcal{A}}$ with sufficiently small diameter consisting of rectangles $P_a \sim I_a^s \times I_a^u$ delimited by compact pieces $I_a^s$, $I_a^u$, of stable and unstable manifolds of certain points of $\Lambda$. We can define a subset $\mathcal{B}\subset \mathcal{A}^2$ of admissible transitions as the subset of pairs $(a,b)$ such that $\varphi(P_a)\cap P_b\neq \emptyset$. In this way there is an homeomorphism $\Pi:\Lambda\rightarrow \Sigma_{\mathcal{B}}$ such that 
	$$\Pi(\varphi(x))=\sigma(\Pi(x)).$$
Next, fix $p\in \Lambda$ a saddle fixed point. Consider the regular Cantor sets $K^s=W^s(p)\cap \Lambda$ and $K^u=W^u(p)\cap \Lambda$ (cf. \cite{PT93}, p. 54). We can define $K^s$ and $K^u$ as follows:

We define $g_s$ and $g_u$ in the following way: If $y\in R_{a_1}\cap \varphi(R_{a_0})$ we put

$$g_s(\pi^u_{a_1}(y))=\pi^u_{a_0}(\varphi^{-1}(y)).$$
and if $z\in R_{a_0}\cap \varphi^{-1}(R_{a_1})$ we put

$$g_u(\pi^s_{a_0}(z))=\pi^s_{a_1}(\varphi(z)).$$
We have that $g_s$ and $g_u$ are $C^{1+\epsilon}$ expanding maps of type $\Sigma_{\mathcal{B}}$ defining $K^s$ and $K^u$ in the sense that

\begin{enumerate}
	\item[(i)] The domains of $g_s$ and $g_u$ are disjoint unions $$\bigsqcup_{(a_0,a_1)\in \mathcal{B}} I^s(a_1,a_0) \ \mbox{and} \bigsqcup_{(a_0,a_1)\in \mathcal{B}}I^u(a_0,a_1),$$ where $I^s(a_1,a_0)$, resp. $I^u(a_0,a_1)$, are compact subintervals of $I^s_{a_1}$, resp. $I^u_{a_0}$;
	
	\item[(ii)] For each $(a_0,a_1)\in \mathcal{B},$ the restrictions $g_s|_{I^s(a_1,a_0)}$ and $g_u|_{|I^u(a_0,a_1)}$ are $C^{1+\epsilon}$ diffeomorphisms onto $I^s_{a_0}$ and $I^u_{a_1}$ with $|Dg_s(t)|,|Dg_u(t)|>1$, for all $t\in I^s(a_1,a_0)$, $t\in I^u(a_0,a_1)$ (for appropriate choices of the parametrization of $I^s_a$ and $I^u_a$);  
	
	\item[(iii)] $K^s$ and $K^u$ satisfies 
	$$K^s=\bigcap_{n\ge 0}g_s^{-n}\left(\bigsqcup_{(a_0,a_1)\in \mathcal{B}} I^s(a_1,a_0)\right) \ \ \ K^u=\bigcap_{n\ge 0}g_u^{-n}\left(\bigsqcup_{(a_0,a_1)\in \mathcal{B}}I^u(a_0,a_1)\right).$$
	
\end{enumerate}
This definition is in accordance with the remark \ref{rr1}.

The stable and unstable Cantor sets, $K^s$ and $K^u$, respectively, are closely related to the fractal geometry of the horseshoe $\Lambda$; for instance, it is well-known that

\begin{equation}\label{eq.1}
HD(\Lambda)=HD(K^s)+HD(K^u)
\end{equation}

Fix a Markov partition $\mathcal{P}=\{P_a\}_{a\in \mathcal{A}}$. Given an admissible finite sequence $\theta=(a_1,...,a_n)\in \mathcal{A}^n$ (i.e., $(a_i,a_i+1)\in \mathcal{B}$) for all $1\le i<n$, we define
	$$I^u(\theta)=\{x\in K^u; g_u^i(x)\in I^u(a_i,a_{i+1}), i=1,2,...,n-1\}.$$
Moreover, in a similar way
	$$I^s(\theta^t)=\{y\in K^s; g_s^i(y)\in I^s(a_{i},a_{i-1}), i=2,...,n\}.$$

We recall that a Markov partition of $\Lambda$ for $\varphi$ is a finite collection $\mathcal{P}=\{P_1,P_2,...,P_k\}$ of boxes, i.e., diffeomorphic images of the square $Q=[-1,1]$, say $P_1=f_1(Q),..., P_k=f_k(Q)$ such that

\begin{enumerate}
	\item[(i)] $\Lambda\subset P_1\cup ... \cup P_k$
	
	\item[(ii)] int$(P_i)\cap $int$P_j=\emptyset$ if $i\neq j$.
	
	\item[(iii)] $\varphi(\partial_sP_i)\subset \bigcup_j \partial_s P_j$ and $\varphi^{-1}(\partial_u P_i) \subset \bigcup_j \partial_u P_j$. %where
	
%	$$\partial_s P_i=f_i(\{(x_s,x_u); -1\le x_s\le 1, |x_u|=1 \}) \  \partial_u P_i=f_i(\{(x_s,x_u); -1\le x_u\le 1, |x_s|=1 \}).$$
	
	\item[(iv)] There is a positive integer $n$ such that $\varphi^n(P_i)\cap P_j\neq \emptyset$, for all $1\le i,j\le k.$
\end{enumerate}

We remark that we can obtain Markov partitions with similar properties as the lemma (\ref{L2}). We have the following theorem about Markov partitions useful for us in many results (cf. \cite{PT93}, pag. 172, theorem 2). We shall use the next theorem many times

\begin{theorem}\label{p}
	If $\Lambda$ is a horseshoe associated to a $C^2$-diffeomorphism $\varphi$, then there are Markov partitions for $\Lambda$ with arbitrarily small diameter.
\end{theorem}

\subsection{Hyperbolic sets of finite type}

Given a Horseshoe $\Lambda$ we know that there is a complete subshift of finite type $\Sigma(\mathcal{B})$ and a homeomorphism $\Pi:\Lambda\rightarrow \Sigma(\mathcal{B})$ such that $\varphi\circ \Pi=\Pi \circ \sigma$ as explained above. We write $\mathcal{B}_n=\{\gamma=(a_0,a_1,...,a_n); (a_i,a_{i+1})\in \mathcal{B}\}$ and $\mathcal{B}^{\infty}=\sqcup_{n\ge 0}\mathcal{B}_n$. If $\gamma^1,\gamma^2\in \mathcal{B}^{\infty}$ the concatenation of $\gamma^1=(a_1^1,...,a^1_{m_1})$ and $\gamma^2=(a_1^2,...,a^2_{m_2})$ is $\gamma^1\gamma^2$ and it is defined if $(a^1_{m_1},a^2_1)\in \mathcal{B}$. In this case, $\gamma^1\gamma^2$ is admissible.

	 Let $X=\{\gamma^1,\gamma^2,...,\gamma^n\}$ with $\gamma^i$ admissible finite word. We write $\gamma^i\to  \gamma^j$ iff $\gamma^i\gamma^j$ is admissible and $\gamma^i\sim \gamma^j$ iff $\gamma^i\to \gamma^j$ and $\gamma^j\to \gamma^i$. %In particular we can define a matriz $T=(b_{\gamma^i,\gamma^j})$ such that $b_{\gamma^i,\gamma^j}=1$ if $\gamma^i\to \gamma^j$ and $b_{\gamma^i,\gamma^j}=0$ otherwise. Consider $\Sigma(X)=\{\theta=(\theta^-\gamma^*\theta^+)\in \Sigma(\mathcal{B}); \gamma\in X\}$.
	 
\begin{definition}[Hyperbolic set of finite type]\label{hsf}
	The pre-image $\tilde{\Lambda}$ of 
	$$\Sigma(X)=\{\theta=(\beta^i)_{i\in \mathbb{Z}}; \beta^i\in X\}$$ by $\Pi$ will be called a \textbf{hyperbolic set of finite type}.
\end{definition}	
By definition, hyperbolic sets of finite type have local product sttucture.
\begin{definition}\label{subhorshoe}
A hyperbolic set of finite type, $\tilde{\Lambda}$, is called a \textbf{subhorseshoe} of $\Lambda$ if $\varphi|{\tilde{\Lambda}}$ is transitive. If $\tilde{\Lambda}$ is a periodic orbit it will be call of \textbf{trivial}.
\end{definition}
%Next we use the symbolic characterization of $\Lambda$ to describe the dynamics restricted to hyperbolic sets of finite type. We recall that given $\mathcal{B}$ there is a transition matrix $B=(b_{ij})_{i,j\in \mathcal{A}}$ such that $b_{ij}=1$ is $(i,j)\in \mathcal{B}$ and $b_{ij}=0$ otherwise. Moreover, since $\Lambda$ is transitive for every $i,j\in \mathcal{A}$ there is $\ell$ such $b^{\ell}_{ij}>0$, where $b^{\ell}_{ij}$ is the element at $i$-th line and $j$-th column of $\ell$-th power of $B$. Of course that for general hyperbolic sets of finite type this not occur.
%\begin{definition}\label{def.equivalence}
%We say that $\gamma^i$ is equivalent to $j$ and we denote $i\sim j$, if and only if there are $\ell$ and $k$ such that $b^{\ell}_{ji}\cdot b^{k}_{ij}>0$.
%\end{definition}

\begin{definition}\label{def.transient}
We call an word $\gamma^i\in X$ of \textbf{transient} if $\gamma^i$ it is not related to itself.
\end{definition}
A relation class determines a submatrix of $B$. The submatrix is composed of all $b_{\gamma^i,\gamma^j}$ with $[\gamma^i]=[\gamma^j]$. This submatrix is called an \textbf{irreducible component} of $B$.  It follows from Theorem 1.3.10 of \cite{BK} that there is a permutation matrix $C$ such that
$$C^{-1}BC=\left[ \begin{array}{ccccc}
B_1 & \ast & \ast & ... & \ast \\
0 & B_2 & \ast & ... & \ast \\
0 & 0 & B_3 & ... & \ast \\
. & . & . & ... & . \\
. & . & . & ... & . \\
. & . & . & ... & . \\
0 & 0 & 0 & ... & B_m
\end{array}  \right]
$$ 
where $B_i$ is either a nonnegative, square, irreducible matrix corresponding an irreducible component of $B$ or the one by one matrix $[0]$ corresponding to a transient state.

Using the homeomorphism $\Pi:\Lambda\rightarrow \Sigma_B$ for any hyperbolic set of finite type $\tilde{\Lambda}$ there is a submatrix $\tilde{B}$ of $B$ such that $\varphi |\tilde{\Lambda}\sim \sigma |\Sigma_{\tilde{B}}$. Following what was said above there are $\tilde{B}_i$ irreducible components such that the preimage of $\Pi$ of the associated sub-shifts $\Sigma_{\tilde{B}_i}$ are subhorseshoes. 
\begin{proposition}\label{prop2}
Let $\tilde{\Lambda}$ be a hyperbolic set of finite type. For $x\in \tilde{\Lambda}$ we have

\begin{enumerate}
\item[(i)] There are subhorseshoes $\Lambda_1$ and $\Lambda_2$ such that $\alpha(x)\subset \Lambda_1$ and $\omega(x)\subset\Lambda_2$.

\item[(ii)] $x$ is nonwandering if and only if $x$ is in a subhorseshoe $\Lambda'$.

\item[(iii)] $x$ is nonwandering if and only if there is a subhorseshoe $\Lambda'$ such that 
	$$\alpha(x)\cup w(x)\subset \Lambda'.$$
\end{enumerate}
\end{proposition}
\begin{proof}
This is the observation 5.1.2 from \cite{BK}.
\end{proof}
By the Proposition \ref{prop2}, if $x\in \tilde{\Lambda}$ is wandering then there are subhorseshoes $\Lambda_1\neq \Lambda_2$ such that $\alpha(x)\subset \Lambda_1$ and $\omega(x)\subset\Lambda_2$. 	
\begin{definition}\label{ts}
 Any set $\bar{\Lambda}$ for which there are subhorshoes $\Lambda_1$ and $\Lambda_2$ such that
	$$\bar{\Lambda}=\cup\{x\in \tilde{\Lambda}; \alpha(x)\subset \Lambda_1 \ \mbox{and} \ \omega(x)\subset\Lambda_2\}$$
will be called a \textbf{transient set} or \textbf{transient component} of $\tilde{\Lambda}$.
\end{definition}

\begin{remark}\label{description1}
	
It follows from the above description that any hyperbolic set of finite type can be written as
	$$\tilde{\Lambda}=\bigcup_{i=1}^k\tilde{\Lambda}_i,$$
where it is possible that some $\tilde{\Lambda}_j$ are periodic orbits, horseshoes or transient sets.

\end{remark}

\begin{remark}\label{o1}
	Note that there are only countably many hyperbolic set of finite type of $\Lambda$. In fact, each Markov partition has only finitely many elements and therefore the number of boxes in all the Markov partition is countable.
\end{remark}

Note that if $\bar{\Lambda}=\cup\{x\in \tilde{\Lambda}; \alpha(x)\subset \Lambda_1 \ \mbox{and} \ \omega(x)\subset\Lambda_2\}$ then
\begin{equation}\label{eq.transient}
HD(\bar{\Lambda})=HD(K^s(\Lambda_1))+HD(K^u(\Lambda_2)),
\end{equation}
where $K^s(\Lambda_1)$ is the stable Cantor set associated to $\Lambda_1$ and $K^u(\Lambda_2)$ is the unstable Cantor set associated to $\Lambda_2$.

\subsection{Fractal Geometry of Hyperbolic sets of finite type}
Recall that every $C^1$-manifold has a $C^{\omega}$-subatlas, so without loss of generality we assume, from now on, $M$ is a real analytic surface.
We want to use the above results about regular Cantor sets on the real line for $K^s$ and $K^u$ to show the following proposition

\begin{proposition}\label{2}
	Let $\varphi:M\rightarrow M$ be a $C^2$-diffeomorphism with a horseshoe $\Lambda_{\varphi}$ and $HD(\Lambda_{\varphi})\ge b>0$. There is a $C^2$-open set $\mathcal{U}\ni \varphi$ and a residual set $\mathcal{R}\subset \mathcal{U}$ such that for $\psi\in \mathcal{R}$ with a horseshoe $\Lambda_{\psi}$, continuation of $\Lambda_{\varphi}$, every hyperbolic set of finite type $\tilde{\Lambda}_{\psi}$ is such that 
	$$HD(\tilde{\Lambda}_{\psi})\neq b.$$		
	\begin{remark}
		Note that if $HD(\Lambda)<b$ the statement is trivial and it is true over the open set $\mathcal{U}$. Moreover, $\mathcal{U}$ is an open set where we have hyperbolic continuation.
	\end{remark}

\end{proposition}	
	Next, we give a simple but central lemma for our goal.
	
	\begin{lemma}\label{L3}
		If we have $C^s$ maps $\varphi:M\rightarrow M$ and $\psi:M \rightarrow M$, such that when restrict to an open set $U\subset M$ are $C^r$ maps, $1\le s\le r\le \omega$ and they are $C^1$-sufficiently close over this one, then there is a $C^r$-isotopy $H:U\times [0,1]\rightarrow M$ such that $H(x,0)=(\varphi|U)(x)$ and $H(x,1)=(\psi|U)(x)$. Moreover $H(\cdot,t):U\rightarrow M$ is $C^1$-close to $\varphi|U$ for all $t\in [0,1]$. 
	\end{lemma}
	
	\begin{proof}
		Let $\delta>0$ be the injectivity radius of $M$. Then, if $\varphi$ and $\psi$ are $\delta$-$C^1$-close then for every $x\in U$ there exists $v(x)\in T_{\varphi(x)}M$ such that $\exp_{\varphi(x)}(v(x))=\psi(x).$ Since $\exp_p$ is a $C^r$ diffeomorphism of $B^{T_pM}_{\delta}(0_p)$ on $B^M_{\delta}(p)$ and $v$ is defined implicitly by $C^r$ functions, over $U$, we have that the map $H:U\times [0,1]\rightarrow M$ defined by
		
		$$H(x,t)=\exp_{\varphi(x)}(tv(x))$$
		is a $C^r$ path between $\varphi$ and $\psi$ as the statement
		\end{proof}
	
	\begin{remark}\label{cons}
		Using generating functions we can obtain the same statement adding that if $\varphi$ and $\psi$ are conservative then we can find a $C^r$-isotopy which is conservative for every $t$.  
	\end{remark}
	%\mathcal{R} is an open and dense set because the subhorshoes have Hausdorff dimension strictly smaller than $HD(\Lambda)$, so, use the lemma (2.2.3).
	% $HD(\Lambda_{a+\epsilon})>HD(\Lambda_a),$ where $a=a(f,\varphi)$, if $f^{-1}(t)$ intersects $\Lambda$ at a point doesn't belong  to $\partial_s\Lambda$ or $\partial_u\Lambda$, where $f\in \mathcal{R}_{\Lambda}$.

We shall use the following version of a result due to Mark Pollicott:

\begin{theorem}\label{T1}
	Let $\varphi_{\mu}:M\rightarrow M$ be a family of $C^1$-diffeomorphisms of $M$, where $(\mu, x)\mapsto (\varphi_{\mu}|U)(x)$ is $C^{\omega}$, $\mu\in (-\delta,\delta)$ and such that $\Lambda_0=\bigcap_{n\in \mathbb{Z}}\varphi^n_0(U)$
	is a horseshoe for $\varphi_0$ and $\Lambda_{\mu}$ denote the hyperbolic continuation of $\Lambda_0$ associated to $\varphi_{\mu}$. Then,
	$$\mu\mapsto HD(\Lambda_{\mu})$$
	is an analytic function. 
\end{theorem}

\begin{proof}
	We observe that by the basic structural stability theorem for Axiom A diffeomorphisms there is a natural bijection between the periodic points of
	$\varphi_0$ and $\varphi_{\mu}$. Moreover, the map $\mu\mapsto p_{\mu}$ the hyperbolic continuation of periodic points is $C^{\omega}$, because $\varphi_{\mu}|U$ is analytic. The claim follows from lemmas 9 and 10 in \cite{P-15}.
\end{proof}
In fact, if we call $K^{s,u}_{\mu}$ the stable and unstable Cantor sets associated to $\Lambda_{\mu}$ then $d_{s,u}({\mu})=HD(K^{s,u}_{\mu})$ are analytical functions. This follows from the proof of the above Theorem in \cite{P-15}.
\begin{lemma}\label{l4}
	If $\Lambda_{\varphi}$ is a horseshoe  associated to a $C^{\omega}$-diffeomorphism $\varphi$, with $HD(\Lambda_{\varphi})=b$. Then, there is a $C^2$- diffeomorphism $\hat{\varphi}$, $C^2$-close to $\varphi$, such that $HD(\Lambda_{\hat{\varphi}})\neq b$, where $\Lambda_{\hat{\varphi}}$ is the hyperbolic continuation of $\Lambda_{\varphi}$.
\end{lemma}
\begin{proof}
	Let $\varphi:M\rightarrow M$ be a $C^{\omega}$-diffeomorphism of a compact connected surface with a horseshoe $\Lambda_{\varphi}$. The foliations of $\Lambda$ may be only $C^{1+\epsilon}$.  We know that there is a Markov partition of $\Lambda_{\varphi}$ such that in each piece $P_i$ of the partition $\mathcal{P}$, in $C^1$ coordinates given by stable and unstable foliations, the diffeomorphism has the form $\varphi(x,y)=(f_i(x),g_i(y))$ at $P_i\in \mathcal{P}$. Taking a fixed point $p\in \Lambda$ and consider the stable regular Cantor set $K^s=W_{\varphi}^s(p)\cap \Lambda$. We know that $K^s$ is a regular Cantor set. We can suppose that its Markov partition has the interior property $(IP)$ as in the lemma \ref{L2}. Now, we replace the foliations by $C^{\omega}$ foliations, $C^1$-close to the previous ones, and change the diffeomorphism in order to have the new invariant foliations in a neighbourhood of the horseshoe, defining it by the $C^{\omega}$ formulas $\tilde{\varphi}(x,y)=(\tilde{f}_i(x),\tilde{g}_i(y))$ in the pieces $\tilde{P}_i$, where $\tilde{\varphi}$ is $C^1$-close to $\varphi$ (cf. \cite{M1}). As we observe in the remark \ref{re 2} the Markov partition associated to now $C^{\omega}$-regular Cantor set $\tilde{K}^s=W_{\tilde{\varphi}}^s(\tilde{p})\cap \Lambda_{\tilde{\varphi}}$, also has the property in the Lemma \ref{L2}, where $(\tilde{p},\tilde{\Lambda})$ are the hyperbolic continuation of $(p,\Lambda)$, $\tilde{p}\in \tilde{\Lambda}$.  %We suppose that $HD(\Lambda)=HD(\tilde{\Lambda})$, otherwise we do for $\Lambda$ and $\tilde{\Lambda}$ the same process described below . 
	By the proposition \ref{P1} there is a $C^w$ map $\hat{f}_i$, $C^2$ close to $f_i$, such that if we put $\hat{\varphi}(x,y)=(\hat{f}_i(x),\tilde{g}_i(y))$ in the corresponding Markov partition for the extended diffeomorphism $\hat{\varphi}:M\rightarrow M$ with the property that $\hat{\varphi}=\varphi$ on $M\backslash U$, where $U$ is the neighbourhood of $\Lambda_{\varphi}$ such that $\Lambda_{\varphi}=\bigcap_{n\in \mathbb{Z}}\varphi^n(U)$, we have either $HD(\hat{K^s})\neq HD(\tilde{K}^s)$ or 
	$HD(\hat{K}^s)\neq HD(K^s)$ where $\hat{K}^s=W_{\hat{\varphi}}^s(\hat{p})\cap \Lambda_{\hat{\varphi}}$. We suppose without loss of generality that $HD(\hat{K}^s)\neq HD(K^s)$. Thus, both diffeomorphisms $\varphi$ and $\hat{\varphi}$ are analytic when restrict to $U$ and by the Lemma \ref{L3} we have a family $\{\varphi_{\mu}\}_{\mu\in [0,1]}$ of $C^{\omega}$ maps, $\varphi_{\mu}:U\rightarrow M$, such that $\varphi_0=\varphi|U$ and $\varphi_1=\hat{\varphi}|U$. Taking $\tilde{f}_i$ and $\tilde{g}_i$ sufficiently $C^1$ close to $f_i$ and $g_i$ and $\hat{f}_i$ sufficiently $C^2$ close to $\tilde{f}_i$, respectively, for every $i$ such that makes sense talk to $\tilde{P}_i$ and $\hat{P_i}$, we can assume that, for a possible small neighbourhood $\tilde{U}\subset U$, because the continuity of $(\mu,x)\mapsto \varphi_{\mu}(x)$, %(but we still denote by $U$)
	$\varphi_{\mu}(\tilde{U})\subset \varphi(U)$, for all $\mu\in [0,1]$. Thus, we can extend $\varphi_{\mu}$ to $M$ with the property that $\varphi_{\mu}=\varphi$ on $M\backslash U$ (cf. \cite{H73}, Lemma 2.8, p.50). Note that $\varphi_1=\hat{\varphi}$ is $C^1$-close to $\varphi$, because inside $U$ we only have $\varphi |U$ $C^1$-close to $\hat{\varphi}|U$ but for $\mu$ close to $0$ we have that $\varphi_{\mu}$ is $C^2$-close $\varphi$, by the Lemma \ref{L3}. Then we have a family of $C^2$-diffeomorphisms satisfying the hypotheses of Theorem \ref{T1} and since the map $\mu\mapsto HD(\Lambda_{\mu})$ is analytic with $HD(\Lambda_0)\neq HD(\Lambda_1)$ there are arbitrarily close to $0$ parameters $\mu$  such that $HD(\Lambda_{\mu})\neq HD(\Lambda_0)=HD(\Lambda_{\varphi})$, because (\ref{eq.1}). But, $\varphi_{\mu}$ is $C^2$-close to $\varphi$ when $\mu$ is close to $0$. 
\end{proof}

\begin{remark}\label{rl4}
	The same statement is true if we replace \textbf{horseshoe} by \textbf{hyperbolic set of finite type}. In fact, by (\ref{eq.transient}), every  hyperbolic set of finite type $\tilde{\Lambda}_{\varphi}$ can be write as a finite union of periodic orbits, horseshoes and transient set
	$$\tilde{\Lambda}_{\varphi}=\tilde{\Lambda}^1_{\varphi}\cup ... \cup \tilde{\Lambda}^k_{\varphi}.$$
	Moreover, $HD(\tilde{\Lambda}_{\varphi})=\max\{HD(\tilde{\Lambda}^i_{\varphi});1\le i\le k\}$.
	Let $j$ such that $$HD(\tilde{\Lambda}_{\varphi})=HD(\tilde{\Lambda}^j_{\varphi}).$$ If $\tilde{\Lambda}^j_{\varphi}$ is a horseshoe there is nothing to do, it follows from the above lemma. If $\tilde{\Lambda}^j_{\varphi}$ is a transient set we proceed as well as the proof of the above lemma, just note that $HD(\tilde{\Lambda}^j_{\varphi})=HD(K^u(\Lambda'))+HD(K ^u(\Lambda''))$, where $\tilde{\Lambda}^j_{\varphi}=W^s(\Lambda'_{\varphi})\cap W^u(\Lambda''_{\varphi})$ and $K^s(\Lambda')$ and $K^u(\Lambda'')$ are the stable Cantor set of $\Lambda'$ and ustable Cantor set of $\Lambda'',$ respectively.
	
\end{remark}

In order to show the Proposition \ref{2} we remember a deep result of Morrey in the compact case and Grauert and Remmert in the general case ( cf. \cite{H73} p. 65 )

\begin{theorem}\label{M}
	Let $M$ and $N$ be $C^{\omega}$ manifolds. Then $C^{\omega}(M,N)$ is dense in $C^r(M,N)$, $1\le r\le \infty$.
\end{theorem}  

It follows from our results in this section and of the Theorem \ref{M} that

\begin{corollary}\label{c1}
	Let $\tilde{\Lambda}_{\varphi}$ be a hyperbolic set of finite type of $\Lambda_{\varphi}$ with $HD(\tilde{\Lambda}_{\varphi})=b$, associated to a $C^2$ diffeomorphism $\varphi$. Then, there is a $C^2$-open set $\mathcal{U}$, which does not depend on $\tilde{\Lambda}_{\varphi}$, such that $\mathcal{U}^{\ast}=\{\psi\in \mathcal{U};HD(\tilde{\Lambda}_{\psi})\neq b\}$ is a $C^2$-open and dense subset of $\mathcal{U}$, where $\tilde{\Lambda}_{\psi}$ denotes the hyperbolic continuation of $\tilde{\Lambda}_{\varphi}$.
\end{corollary}
\begin{proof}
	Let $\mathcal{U}\ni \varphi$ be the open set where we have hyperbolic continuation. Take any open set $\hat{\mathcal{U}}\subset \mathcal{U}$. By the Theorem \ref{M} there is an analytic diffeomorphism$\hat{\varphi}\in \hat{\mathcal{U}}$. If $HD(\tilde{\Lambda}_{\hat{\varphi}})\neq b$ there is nothing to do. Suppose then $HD(\tilde{\Lambda}_{\hat{\varphi}})=b$. By the Lemma \ref{l4} or its remark \ref{rl4} we have a $C^2$-diffeomorphism $\psi\in \hat{\mathcal{U}}$ such that $HD(\tilde{\Lambda}_{\psi})\neq b$. Thus $\mathcal{U}^{\ast}$ is dense. Moreover, in the same notation of remark \ref{rl4}, since $HD(\tilde{\Lambda}_{\psi})=\max\{HD(\tilde{\Lambda}^i_{\psi});1\le i\le k\}$ we have that the map $\psi\mapsto HD(\tilde{\Lambda}_{\psi})$ is a continuous function we have that $\mathcal{U}^{\ast}$ is open.
\end{proof}

\bigskip

\begin{flushleft}
\bf{Proof of Proposition \ref{2}.} 
\end{flushleft}
Let $S(\Lambda_{\varphi})=\{\Lambda^1,\Lambda^2,...\}$ be the set of subhorseshoes of $\Lambda_{\varphi}$. By the Corollary \ref{c1} each $\mathcal{U}^{\ast}_j=\{\psi\in \mathcal{U};HD(\Lambda^j)\neq b\}$ is an open and dense subset of $\mathcal{U}$. Therefore
$$\mathcal{R}=\bigcap_{j\in \mathbb{N}}\mathcal{U}^{\ast}_j$$
is the residual set of the statement.  We observe that in the case $HD(\Lambda_{\varphi})=b$, all the subhorseshoes have Hausdorff dimension smaller than $b$ and we are in the conditions of the Corollary \ref{c1}. Thus, $\mathcal{R}$ is an open and dense subset. So we have proved the proposition \ref{2}.

\begin{corollary}\label{c2}
	Given $b\in (0,2)$, let $\mathcal{U}=\{\varphi\in \Diff^2(M); HD(\Lambda_{\varphi})>b\}$. Then there exists a residual set $\mathcal{R}\subset \mathcal{U}$ such that if $\varphi\in \mathcal{R}$ we have $HD(\tilde{\Lambda})\neq b$ for all hyperbolic set of finite type $\tilde{\Lambda}$ of $\Lambda_{\varphi}$.
\end{corollary}
\begin{proof}
	We know that given $\varphi\in \mathcal{U}$ there is an open set $\mathcal{U}_{\varphi}$ and a residual $\mathcal{R}_{\varphi}\subset \mathcal{U}_{\varphi}$ such that if $\psi\in \mathcal{R}_{\varphi}$ we have $HD(\tilde{\Lambda}_{\psi})\neq b$ for every subhorseshoe of $\Lambda_{\psi}$. Since $\mathcal{U}=\bigcup_{\varphi\in \mathcal{U}}\mathcal{U}_{\varphi}$ is Lindel\"of we can take a countable subcover of $\{\mathcal{U}_{\varphi}\}_{\varphi}$. Then $\mathcal{U}=\bigcup_{j\in \mathbb{N}}\mathcal{U}_{\varphi_j}$. So $\mathcal{R}=\cup_{j\ge 1}\mathcal{R}_{\varphi_j}$ is a residual set in the statement.
\end{proof}

\begin{proposition}\label{cons1}
	Let $\varphi:M\rightarrow M$ be a $C^2$-conservative diffeomorphism with a horseshoe $\Lambda_{\varphi}$ and $HD(\Lambda_{\varphi})\ge b>0$. There is a $C^2$-open set $\mathcal{U}\ni \varphi$ of conservative diffeomorphisms and a residual set $\mathcal{R}\subset \mathcal{U}$ such that for $\psi\in \mathcal{R}$ with a horseshoe $\Lambda_{\psi}$, continuation of $\Lambda_{\varphi}$, every hyperbolic set of finite type $\tilde{\Lambda}_{\psi}$ have Hausdorff dimension different from $b$.
\end{proposition}
\begin{proof}
	The proof is analogous to the proof of proposition \ref{2}. We need to prove a version of lemma \ref{l4}, which follows from the remark \ref{cons} and the analogous theorem on the density of analytic conservative diffeomorphisms, cf \cite{GL}.
\end{proof}

\section{The Markov and Lagrange dynamical spectra}

We begin this section giving all the ingredients outside this paper. Recall that by the section 3, if $p$ is a fixed point in $\Lambda$, the sets $K^s=W^s(p)\cap \Lambda$ and $K^u=W^u(p)\cap \Lambda$, are regular Cantor sets defined by the expanding maps $g_s$ and $g_u$.
\begin{theorem}[MY-10]\label{MY-10}
	Suppose that the sum of the Hausdorff dimensions of the regular Cantor sets $K^s$ and $K^u$, defined by $g_s$ and $g_u$, is greater than one. If the neighbourhood $\mathcal{U}$ of $\varphi$ in $\Diff^{\infty}(M)$ is sufficiently small, there is an open and dense set $\mathcal{U}^{\ast}\subset \mathcal{U}$ such that, for $\psi\in \mathcal{U}^{\ast}$, the corresponding pair of expanding $(g,\tilde{g})$ belongs to $V$ ( cf. \cite{MY-10} and \cite{MY-01}).
\end{theorem}

Using \cite{MY-10}, we can write the main theorem in \cite{MR-15} in the following way

\begin{theorem}[MR-15]\label{MR-15}
	Let $\Lambda$ be a horseshoe associated to a $C^2$-diffeomorphism $\varphi$ such that $HD(\Lambda)>1$. Then, there is 
	a $C^2$-neighbourhood $\mathcal{U}\ni \varphi$ and an open and dense subset $\mathcal{U}^{\ast}(\Lambda)$ such that, if $\Lambda_{\psi}$ denote the hyperbolic continuation of $\Lambda$ associated to $\psi\in \mathcal{U}^{\ast}(\Lambda)$, there is an open and dense set $H_{\psi}(\Lambda_{\psi})\subset C^k(M,\mathbb{R})$, $k\ge 1$, such that for all $f\in H_{\psi}(\Lambda_{\psi})$
	$$\inte(L_f(\Lambda_{\psi}))\neq \emptyset$$
	
	%arbitrarily close to $\varphi$, there is a diffeomorphism $\varphi_0$ and a $C^2$-neighbourhood $\mathcal{U}\ni \varphi_0$ such that, if $\Lambda_{\psi}$ denotes the hyperbolic continuation of $\Lambda$ associated to $\psi\in \mathcal{U}$, there is an open and dense set $H_{\psi}(\Lambda)\subset
\end{theorem}

%Recall that $M$ is a real analytic surface. We have the following

\begin{corollary}\label{c}
	Let $\mathcal{U}=\{\varphi\in \Diff^2(M); HD(\Lambda_{\varphi})>1\}$. There is a residual set $\mathcal{R}\subset \mathcal{U}$ such that if $\psi\in \mathcal{R}$ and $\Lambda_{\psi}$ is a horseshoe of $\psi$ with $HD(\Lambda_{\psi})>1$ and $k\ge 1$, then there is a residual set $H_{\psi}\subset C^k(M,\mathbb{R})$, such that for every $f\in H_{\psi}$ we have
	$$\inte(L_f(\tilde{\Lambda}_{\psi}))\neq \emptyset$$
	and
	$$\inte(M_f(\tilde{\Lambda}_{\psi}))\neq \emptyset,$$
	for all subhorseshoe $\tilde{\Lambda}_{\psi}$ of $\Lambda_{\psi}$ such that $HD(\tilde{\Lambda}_{\psi})>1$. 
\end{corollary}

\begin{proof} 
	Firstly we observe that %if we have $\varphi\in \mathcal{U}$ with a horseshoe $\Lambda_{\varphi}$ 
	there are only a countable many choice of the combinatorics of a subhorseshoe, that is, its associated subshift. Fix any combinatorics $c$. For each $\varphi$ we can associated $\Lambda^c_{\varphi}$ the horseshoe with this combinatorics. %Consider $H_{\varphi}=\bigcap_{j\in \mathbb{N}}H_{\varphi}(\Lambda^j)$, where $\Lambda^1, \Lambda^2,...$ are the subhorseshoe of $\Lambda_{\varphi}.$
	
	\begin{flushleft}
		\textbf{Claim.} There is a residual set $\mathcal{R}^c$ such that for $\varphi\in \mathcal{R}^c$ and $HD(\Lambda^c_{\varphi})>1$ there is an open and dense set $H_{\varphi}(\Lambda^c_{\varphi})\subset C^k(M,\mathbb{R})$ we have that
		$$\inte(L_f(\Lambda^c_{\varphi}))\neq \emptyset$$
		and
		$$\inte(M_f(\Lambda^c_{\varphi}))\neq \emptyset.$$
	\end{flushleft}
	If $HD(\Lambda^c_{\varphi})>1$ for any $\varphi$, by the Theorem \ref{MR-15} there is an open $\mathcal{U}_{\varphi}\ni \varphi$ and a residual set $\mathcal{R}_{\varphi}\subset \mathcal{U}_{\varphi}$ such that the claim is true. Take a countable cover of $\{\varphi\in \mathcal{U}; HD(\Lambda^c_{\varphi})>1\}$, say $\mathcal{U}_{\varphi_1}\cup \mathcal{U}_{\varphi_2}\cup...$. So $\mathcal{R}^c_1=\bigcup_{j\ge 1}\mathcal{R}_{\varphi_j}$ is a residual subset of $\{\varphi\in \mathcal{U}; HD(\Lambda^c_{\varphi})>1\}$. 
	Take $\mathcal{R}^c=\{\varphi\in \mathcal{U}; HD(\Lambda^c_{\varphi})<1\}\cup \mathcal{R}^c_1$. This prove the claim. By Corollary \ref{c2} there is a residual set in $\mathcal{U}$ such that there is no subhorseshoe with Hausdorff dimension $1$. Since there are only a countable many choice of the combinatorics of subhorseshoe and countable intersection of residual sets is residual we have the proof of the corollary.
\end{proof}

We will need of a version of the above theorem and of the corollary \ref{c} for hyperbolic sets of finite type. We summarize some facts on the construction to the proof of the theorem \ref{MR-15}. Firstly consider $\mathcal{M}_f(K)=\{z\in K; f(z)\ge f(y), \ \forall y\in K\}.$ Next we put
$$H_{\varphi}(\tilde{\Lambda})=\{f\in C^1(M,\mathbb{R});  \mathcal{M}_{\varphi}(\tilde{\Lambda})=\{z\} \ \mbox{and} \ Df(z)e^{s,u}_z\neq 0\},$$ 
where $\tilde{\Lambda}$ is a hyperbolic set of finite type contained in a horseshoe $\Lambda$. 

\begin{enumerate}
		
	\item[(1)] 	If $f\in C^1(M,\mathbb{R})$ and $\mathcal{M}_{\varphi}(\tilde{\Lambda})=\{z\}$ then $z\in \partial_s\Lambda\cap \partial_u \Lambda$.
	
	\smallskip
	
	\item[(2)] We have that $$H_{\varphi}(\tilde{\Lambda})=\{f\in C^1(M,\mathbb{R});  \mathcal{M}_{\varphi}(\tilde{\Lambda})=\{z\} \ \mbox{and} \ Df(z)e^{s,u}_z\neq 0\},$$ 
	is open and dense.
	
	\smallskip
	
	\item[(3)] If $\tilde{\Lambda}$ is a hyperbolic set of finite type for $\varphi$ and $f\in H_{\varphi}(\tilde{\Lambda})$ then, given $\epsilon>0$ there are $j_0\in \mathbb{N}$, a rectangle $R$ and a diffeomorphism $A$ defined on a neighbohood of $\tilde{\Lambda}_{\varphi}\cap R$ such that
	$$f(\varphi^j(A(\tilde{\Lambda}_{\varphi}\cap R)))\subset M_{f}(\tilde{\Lambda}_{\varphi})$$ and  
	$$HD(\varphi^j(A(\tilde{\Lambda}_{\varphi}\cap R))>HD(\tilde{\Lambda}_{\varphi})-\epsilon.$$
	Moreover,
	$$Df(\varphi^j(A(z)))e^{s,u}(z)\neq 0.$$
	
	\smallskip
	
	\item[(4)]  	If $\tilde{\Lambda}$ is such that $(K^s(\Lambda'),K^u(\Lambda''))\in V$ then 
	$$\inte f(\tilde{\Lambda})\neq \emptyset.$$ 
	
\end{enumerate}

The proofs are analogous to the ideas of proof of Theorem \ref{MR-15}, for instance, note that $\tilde{\Lambda}$ also has the local product structure and from this it follows (1). 

Following this, we have a slightly different form of Moreira-Roma\~na's Theorem

\begin{theorem}\label{MR-adapted}
	Let $\Lambda$ be a horseshoe associated to a $C^2$-diffeomorphism $\varphi$ such that $HD(\Lambda)>1$. Then, there is 
	a $C^2$-neighbourhood $\mathcal{U}\ni \varphi$ and an open and dense subset $\mathcal{U}^{\ast}(\Lambda)$ such that, if $\Lambda_{\psi}$ denotes the hyperbolic continuation of $\Lambda$ associated to $\psi\in \mathcal{U}^{\ast}(\Lambda)$, there is an open and dense set $H_{\psi}(\Lambda_{\psi})\subset C^k(M,\mathbb{R})$, $k\ge 1$, such that for all $f\in H_{\psi}(\Lambda_{\psi})$
	$$\inte(M_f(\Lambda_{\psi}))\neq \emptyset.$$
\end{theorem}

\begin{corollary}\label{c1}
	Let $\mathcal{U}=\{\varphi\in \Diff^2(M); HD(\Lambda_{\varphi})>1\}$. There is a residual set $\mathcal{R}\subset \mathcal{U}$ such that if $\psi\in \mathcal{R}$ and $\Lambda_{\psi}$ is a horseshoe of $\psi$ with $HD(\Lambda_{\psi})>1$ and $k\ge 1$, then there is a residual set $H_{\psi}\subset C^k(M,\mathbb{R})$, such that for every $f\in H_{\psi}$ we have
	$$\inte(M_f(\tilde{\Lambda}_{\psi}))\neq \emptyset,$$
	for all hyperbolic set of finite type $\tilde{\Lambda}_{\psi}$ of $\Lambda_{\psi}$ such that $HD(\Lambda_{\psi})>1$. 
\end{corollary}

\section{Phase Transition Theorem for the Markov Spectrum}

Recall that the Lagrange spectrum of a continuous function $f:M\rightarrow \mathbb{R}$ and a $C^2$-diffeomorphism $\varphi:M\rightarrow M$ associated to a compact invariant set $K$ is defined by
$$L_{f}(K)=\{\ell_{f,\varphi}(x);x\in K\},$$
where $\ell_{f,\varphi}(x)=\limsup_{n\to \infty}f(\varphi^n(x))$.
The set $$M_f(K)=\{m_{f,\varphi}(x);x\in K\},$$ 
where $m_{f,\varphi}(x)=\sup_{n\in \mathbb{Z}} f(\varphi^n(x))$, is called Markov spectrum of $(f,K)$ and it is closely related to $L_f(K)$. We have, by example, that $L_f(K)\subset M_f(K)$ (cf. \cite{MR-15}).

Suppose that $\Lambda$ is a horseshoe associated to $\varphi$. We define for $t\in \mathbb{R}$
$$\Lambda_t=\bigcap_{n\in \mathbb{Z}}\varphi^n(\Lambda\cap f^{-1}(-\infty,t])).$$
%Remember (see chapter 1), that we can imagine that $\Lambda_t$ is the horseshoe under the viewpoint of $f$ up the time $t$, but $\Lambda_t$ may not be a horseshoe. Each continuous function sees the horseshoe at time $t$ in its own way. The Lagrange dynamical spectrum helps us to locate the $\omega$ limit set of points in $\Lambda$. The sets $\Lambda_t$ are close related with the Markov and Lagrange dynamical spectra. The main theorem in this section and the theorem in the next section helps us to understand the relation between $L_f(\Lambda)$ and $\Lambda_t$. We begin by showing the relation with some remarks.

\begin{remark}\label{r3.1} We observe that if $F$ is a closed invariant set of $M$ then 
	$$L_f(F)\subset f(F).$$ In fact, if $x\in F$ then $\varphi^n(x)\in F$ for all $n\in \mathbb{Z}$ and so $f(\varphi^n(x))\in f(F)$, since $f(F)$ is a compact set on the real line, it follows that $\ell_{f,\varphi}(x)=\limsup_nf(\varphi^n(x))\in f(F)$. Next 
	we note that 
	$$L_f(\Lambda)\cap (-\infty,t] \subset f(\Lambda_t).$$
	Since $L_f(\Lambda)\subset M_f(\Lambda)$ it is sufficient to show that $$M_f(\Lambda)\cap (-\infty, t]\subset f(\Lambda_t).$$
	Take $m\in M_f(\Lambda)\cap (-\infty,t]$. Thus, $m=\sup f(\varphi^n(x))\le t$, $x\in \Lambda$. So, $f(\varphi^n(x))\le t$, $\forall n\in \mathbb{Z}$. Therefore $\varphi^n(x)\in f^{-1}((-\infty,t])\cap \Lambda, \forall n\in \mathbb{Z}$. This implies that $x\in \Lambda_t$. Since $\varphi(\Lambda_t)=\Lambda_t$ we have $\omega(x)\cup \alpha(x)\subset \Lambda_t$, where $\omega(x)$ and $\alpha(x)$ are the  $\omega$-limit set and $\alpha$-limite set of $x$, respectively. Thus, $m\in f(\Lambda_t)$.
\end{remark}

Given a $C^2$-diffeomorphism $\varphi:M\rightarrow M$, with a horseshoe $\Lambda_{\varphi}$ such that $HD(\Lambda_{\varphi})>1$ and $f\in C^k(M,\mathbb{R})$, $k\ge 1$, let $a(f,\varphi)=\sup\{t\in \mathbb{R}; HD(\Lambda_t)<1\}$ be the \textbf{ Markov's phase transition parameter} associated $(f, \varphi,\Lambda_{\varphi})$. We observe that, by the remark \ref{r3.1}, 
$$\inte{(M_f(\Lambda_{\varphi})\cap (-\infty,t))}=\emptyset,$$ 
for every $t<a$, because $HD(f(\Lambda_t))\le HD(\Lambda_t)<1$.
% In particular, we also have that
%$$\inte{(L_f(\Lambda_{\varphi})\cap (-\infty,t))}=\emptyset,$$ 
%because $L_f(\Lambda)\subset M_f(\Lambda)$.
%We want to show that 

\begin{remark}\label{r3.2}
	For every $t\in \mathbb{R}$ and a continuous function $f:M\rightarrow \mathbb{R}$ we have that
	$$L_f(\Lambda_t)\subset L_f(\Lambda)\cap (-\infty,t] \ \mbox{and} \ M_f(\Lambda_t)\subset 
	M_f(\Lambda)\cap (-\infty,t].$$
	In fact, if $\ell\in L_f(\Lambda_t)$ then there is $x\in \Lambda_t$ such that $\ell=\ell_{f,\varphi}(x)$. Since $\Lambda_t$ is $\varphi$-invariant, we have $\varphi^n(x)\in \Lambda_t$, $\forall \ n\in \mathbb{Z}$. In particular, $f(\varphi^n(x))\leq t$. Hence $\ell\le t$. It is analogous to $M_f(\Lambda_t)\subset M_f(\Lambda)\cap (-\infty,t].$
	
\end{remark}
Fix a $C^k$-function $f$ from $M$ to $\mathbb{R}$, $k\ge 1$. Consider $\mathcal{U}$ be the open set and $\mathcal{R}\subset \mathcal{U}$ be the residual set as in the proposition \ref{2}, where $b=1$. Let $\Lambda=\Lambda_{\psi}$ be the hyperbolic continuation of $\Lambda_{\varphi}$. When $\psi\in \mathcal{R}$ we write $a=a(f,\psi)$.%Since $\# \partial \Lambda=\#(\partial_s \Lambda \cap \partial_u \Lambda)=\#\mathbb{N}$, given $t\in \mathbb{R}$ such that $f^{-1}(t)\cap \partial\Lambda\neq \emptyset$ we can suppose that $f$ is such that $f^{-1}(t)\cap \partial \Lambda=\emptyset$ and $f^{-1}(t)\cap \Lambda\neq \emptyset$. 

\begin{proposition}
	For every $\psi\in \mathcal{R}$ we have $f^{-1}(a(f,\psi))\cap \Lambda\neq \emptyset$.
\end{proposition}
\begin{proof}
	By contradiction, if $f^{-1}(a)\cap \Lambda=\emptyset$ then $d(f^{-1}(a),\Lambda)>0$, the distance between $f^{-1}(a)$ and $\Lambda$, and $a=a(f,\psi)$. By the continuity of $f$ there exist an $\delta_0>0$ such that $d(f^{-1}(a+\delta),\Lambda)>0$ for all $0<\delta< \delta_0$. Then, there is a Markov Partition $\mathcal{P}$ of $\Lambda$ such that $d(f^{-1}(a+\delta),\tilde{\mathcal{P}})>0$, $0<\delta< \delta_0$ for all Markov partition $\tilde{\mathcal{P}}\prec \mathcal{P}$, because there are Markov partition whose elements have arbitrarily small diameter. Therefore, the elements of $\tilde{\mathcal{P}}$ in $f^{-1}((-\infty,a+\delta))\cap \tilde{\mathcal{P}}$ are the same, independently of $0<\delta< \delta_0$. This implies that $\Lambda_a=\Lambda_{a+\delta}$ and it is a hyperbolic set of finite type contained at $\Lambda$. Since $HD(\tilde{\Lambda})\neq 1$ for all hyperbolic set of finite type of $\Lambda$, we must have $HD(\Lambda_a)<1$.
	Moreover, $f^{-1}((-\infty,a+\delta))\cap \tilde{\mathcal{P}}$ contains $\Lambda_a$. Since $\Lambda_{a+\delta}=\Lambda_a$ we have a contradiction with the definition of $a$.
	
	\includegraphics[scale=0.3, trim=0cm 12.31cm 0cm 0.7cm]{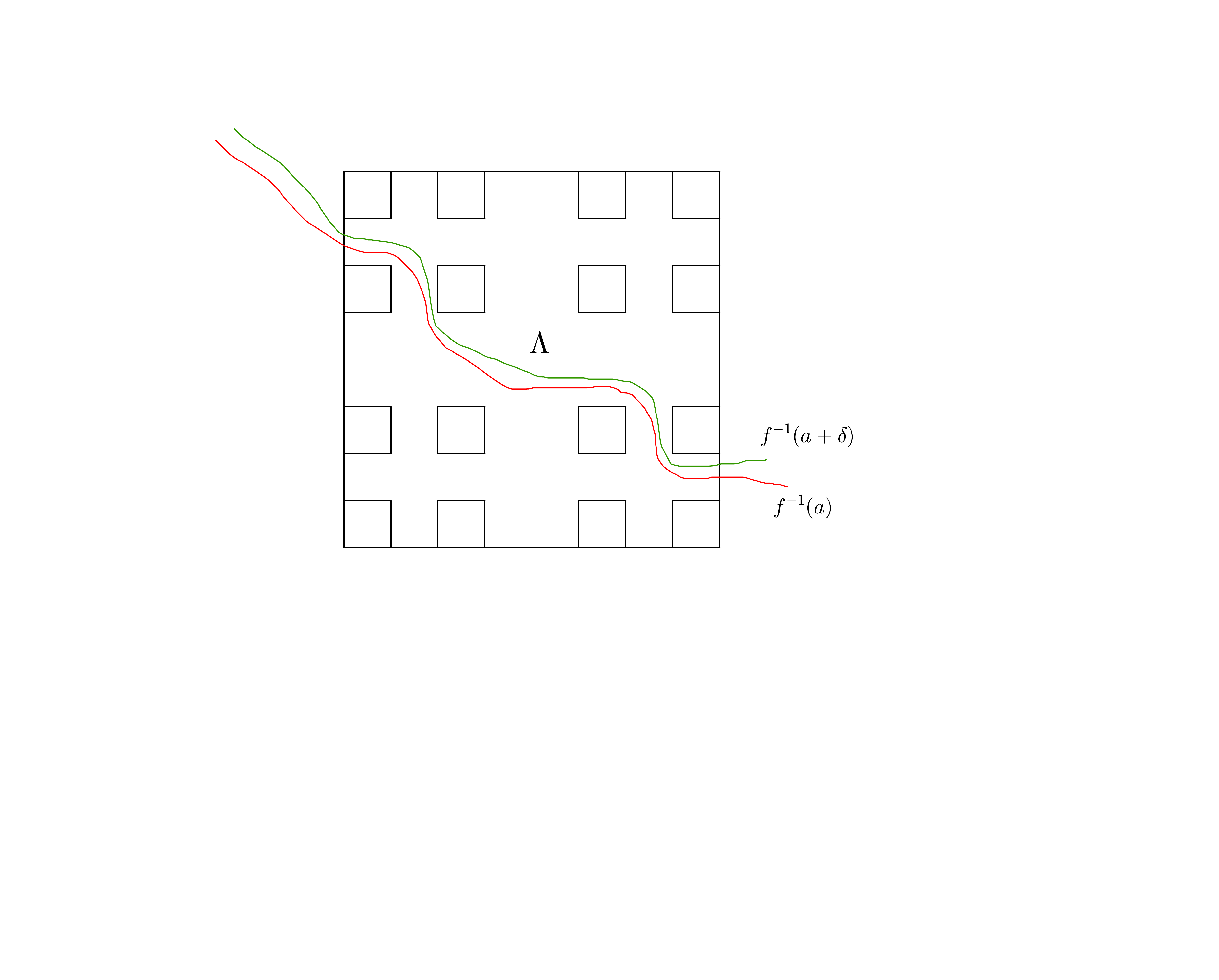}

\end{proof}

%We observe that if $\delta_1> \delta_2$ then $A(\delta_2)\subset A(\delta_1)$.
%The two next lemmas characterize the set $A$.

\begin{proposition}\label{a1}
	If $A=\{\delta\in (0,+\infty]; f^{-1}(a+\delta)\cap \Lambda\neq \emptyset\}$ then $0$ is an accumulation point of $A$. %In particular there is a sequence $\delta_n\to 0$ such that there is $x_n\in f^{-1}(a+\delta_n)\cap \Lambda$ but $x_n\notin \partial \Lambda$
\end{proposition}
\begin{proof}
	By contradiction, suppose that $0$ is isolated. Therefore, there is $\epsilon>0$ such that 
	$$f^{-1}(a+\delta)\cap \Lambda=\emptyset, \forall \delta\in (0,\epsilon).$$
	If $x\in \Lambda_{a+\delta}$ is such that there is $n\in \mathbb{Z}$ with $a<f(\psi^n(x))\le a+\delta$, hence we have $\psi^n(x)\in f^{-1}((a+\delta_1))\cap \Lambda$, for some $\delta_1<\delta$, contradiction, because $\delta_1\notin A$. By the same arguments above we have that $\Lambda_a$ is a subhorshoe with $HD(\Lambda_a)<1$ and  $\Lambda_a=\Lambda_{a+\delta}$ when $\delta < \epsilon$, contradiction with the definition of $a$.
\end{proof}

%\begin{lemma}
%For all $\delta>0$ there is $\tilde{\delta}\in A(\delta)$ such that there is  $x\in f^{-1}(a+\tilde{\delta})\cap \Lambda$ but $x\notin \partial \Lambda$.
%\end{lemma}
%\begin{proof}
%Suppose that there exist $\delta>0$ such that $f^{-1}(a+\tilde{\delta})\subset \partial \Lambda$, for $\tilde{\delta}\in A(\delta)$. Note that this implies

%$$\bigcup_{0<\tilde{\delta}\le \delta}(f^{-1}(a+\tilde{\delta})\cap \Lambda)=f^{-1}((a,a+\delta]))\cap \Lambda\subset \partial \Lambda.$$
%Therefore
%$$f^{-1}((-\infty,a+\delta])\cap \Lambda\subset (f^{-1}((\infty,a]))\cap \Lambda)\cup \partial \Lambda.$$
%Since $\varphi(\partial\Lambda)=\partial \Lambda$ we have that $\Lambda_{a+\tilde{\delta}}\subset \Lambda_a \cup \partial \Lambda$. Since $\#\partial \Lambda=\# \mathbb{N}$ we have $HD(\Lambda_{a+\tilde{\delta}})=HD(\Lambda_a)$. Contradiction.
%\end{proof}
The following proposition is fundamental for our goal and it explains the reason of the name of $a(f,\psi)$.
\begin{proposition}\label{p1} 
	If $\psi\in \mathcal{R}$ then 
	$$a(f,\psi)=\inf\{t\in \mathbb{R}; HD(\Lambda_t)>1\}.$$
\end{proposition}

\begin{proof}
	Take $\delta\in A$. 
	We have $d(f^{-1}(a),f^{-1}(a+\delta))>0$, because $f^{-1}(a)$ and $f^{-1}(a+\delta)$ are disjoint compacts sets. By the Theorem \ref{p} we can take a Markov partition $\mathcal{P}$ with sufficiently small diameter such that 
	\begin{equation}\label{eq3}
	f^{-1}((-\infty,a])\cap \Lambda\subset f^{-1}((-\infty,a])\cap (\cup_{P\in \mathcal{P}} P)\subsetneq f^{-1}((-\infty,a+\delta))\cap (\cup_{P\in \mathcal{P}} P).
	\end{equation}
	By shrinking the diameter of $\mathcal{P}$ if necessary,, we can take an open neighbourhood of elements of $\mathcal{P}$ which intersects $\Lambda_a$ but which has a positive distance from $f^{-1}(a+\delta)$. By (\ref{eq3}) we have that the maximal invariant set of this neighbourhood is a hyperbolic set of finite type, say $\tilde{\Lambda}$, such that
	\begin{eqnarray}\label{eq4}
	\Lambda_a\subset \tilde{\Lambda}\subsetneq \Lambda_{a+\delta}.
	\end{eqnarray}
	In the same way as \ref{eq4}, if $\delta'< \delta$ with $\delta'\in A$, we can find another hyperbolic set of finite type $\Lambda'$ such that
	$$\Lambda_a\subset \Lambda'\subsetneq \Lambda_{a+\delta'}\subset \tilde{\Lambda}\subsetneq \Lambda_{a+\delta}.$$
	Since $\psi\in \mathcal{R}$ we must have  $HD(\tilde{\Lambda})>1$, because $HD(\Lambda_a)\le HD(\tilde{\Lambda})$. Thus, for all $\delta>0$ we have that $HD(\Lambda_{a+\delta})>1$. This implies that  
	$$a=\inf\{t\in \mathbb{R}; HD(\Lambda_t)>1\}$$
	%and $HD(\Lambda_a)=1$. In fact, if $a=a(f,\psi)\neq b(f,\psi)=\inf\{t\in \mathbb{R}; HD(\Lambda_t)>1\}$ then given $a<c<b$ we have by definitions of $a$ and $b$ that $\Lambda_c=1$. But, by (\ref{eq4}) this implies that there is a hyperbolic set of finite type $\tilde{\Lambda}$ with $HD(\tilde{\Lambda})=1$.  
\end{proof}

%\begin{corollary}
%	If $\psi\in \mathcal{R}$ then $\Lambda_a$ is not a hyperbolic set of finite type of $\Lambda$, where $a:=a(f,\psi)$.
%\end{corollary}

Now we are ready to prove

\begin{mydef}
	Given $\varphi\in$ $\Diff^2(M)$ with a horseshoe $\Lambda$ and $HD(\Lambda)>1$, there is an open set $\mathcal{U}\ni \varphi$ and a residual set   $\mathcal{R}\subset \mathcal{U}$ such that for every $\psi\in \mathcal{R}$ and $k\ge 1$ there exists a $C^k$-residual subset %open and dense subset 
	$\mathcal{H}_{\psi}\subset C^k(M,\mathbb{R})$ such that for $f\in \mathcal{H}_{\psi}$, we have
	$$\inte(M_f(\Lambda_{\psi})\cap (-\infty,a+\delta))\neq \emptyset,$$
	for all $\delta>0$, where $a=a(f,\psi)$.
\end{mydef}
\begin{proof}
	Let $\varphi$ be a $C^2$-diffeomorphism with a horseshoe $\Lambda$ with $HD(\Lambda)>1$. By the Corollary \ref{c1} there are an open set $\mathcal{U}\ni \varphi$ and a residual set $\mathcal{R}_1\subset \mathcal{U}$ with a nice properties for us; if $\psi\in \mathcal{R}_1$ there is a residual set $H_{\psi}$ such that
	\begin{eqnarray}\label{eq11}
	\inte{(M_f(\tilde{\Lambda}_{\psi}))}\neq \emptyset.
	\end{eqnarray}
	for every hiperbolic set of finite type $\tilde{\Lambda}_{\psi}$ of $\Lambda_{\psi}$, the hyperbolic continuation of $\Lambda$.
	Let $\mathcal{R}_2$ be the residual set of Proposition \ref{2}. By the proposition \ref{p1} if $\psi\in \mathcal{R}_2$ then
	\begin{eqnarray}\label{eq12}
	HD(\Lambda_{a+\delta})>1.
	\end{eqnarray}
	
	Consider $\mathcal{R}=\mathcal{R}_1\cap \mathcal{R}_2$. So $\mathcal{R}$ is a residual subset and every $\psi\in \mathcal{R}$ has the properties (\ref{eq11}) and (\ref{eq12}). By the remark \ref{r3.2} we have
	\begin{eqnarray}\label{eq13}
	M_f(\Lambda^{\psi}_{a+\delta})\subset M_f(\Lambda_{\psi})\cap (-\infty,a+\delta],
	\end{eqnarray}
	where 
	$$\Lambda^{\psi}_{t}=\bigcap_{n\in \mathbb{Z}}\psi^n(\Lambda_{\psi}\cap f^{-1}((-\infty,t])).$$
	Then, if $\psi\in \mathcal{R}$, the proof of proposition (\ref{p1}) gives us a hyperbolic set of finite type $\tilde{\Lambda}^{\psi}$ with
	$$\Lambda^{\psi}_a\subset \tilde{\Lambda}_{\psi}\subset \Lambda^{\psi}_{a+\delta}.$$
	But this implies that $$\emptyset\neq \inte{(M_f(\tilde{\Lambda}^{\psi}))}\subset \inte{(M_f(\Lambda^{\psi}_{a+\delta}))}.$$
	Hence, by \ref{eq13} we have that
	$$\inte{(M_f(\Lambda_{\psi})\cap (-\infty,a+\delta))}\neq \emptyset.$$
	and the theorem is proven.
\end{proof}	

\begin{mydef3} We have that
	$$\sup\{t\in \mathbb{R}; HD(\Lambda_t)<1\}=\inf\{t\in \mathbb{R};\inte(M_f(\Lambda)\cap (-\infty,t+\delta))\neq \emptyset, \delta>0\}.$$
\end{mydef3}

\begin{proof}
It is immediate from the above theorem.
\end{proof}

%The idea for behind this is the following. We begin showing that there is a residual subset $\mathcal{H}_1$ of $C^2(M^2, \mathbb{R})$ such that $HD(\tilde{\Lambda})\neq 1$ for every $\varphi\in \mathcal{H}$ and  every subhorshoe of $\varphi$.  From this we take a Markov partition, whose elements have diameter sufficiently small, such that we can find a subhorshoe $\tilde{\Lambda}$ between $\Lambda_t$ and $\Lambda_{t+\epsilon}$, so if $HD(\Lambda_t)=HD(\Lambda_{t+\epsilon})$ 
%then $1=HD(\Lambda_ a)=HD(\tilde{\Lambda})=HD(\Lambda_{a+\epsilon})$. But generically this not holds.
\section{Phase Transition Theorem for the Lagrange Spectrum}

The aim of this section is to prove the Theorem B. The proof follows essentially the same ideias of the Theorem A. Here we suppose that $\varphi$ is in the residual set $\mathcal{R}$ such that there are no subhorshoe of dimension $1$ and the corollary \ref{c} and $f\in H_{\varphi}$ as the corollary \ref{c}.

To prove the Phase Transition Theorem for the Lagrange Spectrum, we fix a Markov partition $\mathcal{P}.$ For each $n\in \mathbb{N}$, we consider $\tilde{\mathcal{P}_t}^n\subset \mathcal{P}^n$ the rectangles of Markov partition which cover $\Lambda_t$. For each $n$ let
$$\tilde{\Lambda}_n(t)=\bigcap_{k\in \mathbb{Z}}\varphi^k \left( \Lambda\cap \bigcup_{P\in \mathcal{P}^n_t}P\right),$$
be the hyperbolic set of finite type defined by $\tilde{\mathcal{P}}^n_t.$ We fix $t$ such that there are subhorshoes in $\tilde{\Lambda}_n(t)$. Let $\Lambda'_n(t)$ be the horseshoe contained in $\tilde{\Lambda}_n(t)$ with maximal Hausdorff dimension. Since $\Lambda'_{n+1}(t)\subset \Lambda'_n(t)$ we have that 
$$0< HD(\Lambda'_{n+1}(t))\leq HD(\Lambda'_{n}(t))<2.$$
In particular, there exists $D(t)=\lim_{n\to \infty} HD(\Lambda'_n(t)).$ We observe that $D(t)$ does not depends on the Markov partition $\mathcal{P}$. Let us consider
$$\tilde{a}(f,\varphi):=\sup\{t\in \mathbb{R};D(t)<1\},$$
\textbf{the Lagrange Phase Transition parameter}. It follows from this definition that $\tilde{a}(f,\varphi)\ge a(f,\varphi)$.

For every $\delta>0$ we have that there is a hyperbolic set of finite type, say $\tilde{\Lambda}$, such that $$\Lambda_{\tilde{a}}\subset\tilde{\Lambda}\subset\Lambda_{\tilde{a}+\delta}.$$
But, by definition there is a horseshoe $\Lambda'\subset \tilde{\Lambda}$ with $HD(\Lambda')>1$. In the same way of the proof of Theorem A we have

	\begin{eqnarray}\label{eq23}
L_f(\Lambda_{\tilde{a}+\delta})\subset L_f(\Lambda)\cap (-\infty,a+\delta],
\end{eqnarray}
and since by the corollary \ref{c} we have
	\begin{eqnarray}\label{25}
	\emptyset \neq L_f(\Lambda')
	\end{eqnarray}
we must be have that

$$\inte(L_f(\Lambda)\cap (-\infty,\tilde{a}+\delta))\neq \emptyset.$$
We need to show that $\L(L_f(\Lambda)\cap (-\infty,t))=0$ if $t<a$. For this purpose, we prove the following lemma

\begin{lemma}\label{lagrange1}
Let $\tilde{\Lambda}$ be a hyperbolic set of finite type, 
$$\tilde{\Lambda}=\Lambda_1\cup ... \Lambda_{r}\cup \Lambda_{r+1}\cup ... \Lambda_s$$
where for $1\le i\le r$ we have horseshoes or periodic orbits and for $r+1\le i\le s$ we have transient sets. If $\Lambda'=\bigcup_{i=1}^r \Lambda_i$ then
$$L_f(\tilde{\Lambda})=L_f(\Lambda').$$
\end{lemma}

%\begin{lemma}
%If $\tilde{\Lambda}_n(t)=\tilde{\Lambda}^1_n(t)\cup ... \cup \tilde{\Lambda}^{k_n}_n(t)$ where $\tilde{\Lambda}^{j}_n(t)$ are horseshoes or periodic orbits for $l_n\le j\le k_n$ then
%$$L_f(\Lambda)\cap (-\infty, t)\subset f(\cup_{l_j\le j\le k_n}\tilde{\Lambda}^j_{n}(t)).$$
%\end{lemma}
\begin{proof}
In fact, if $x$ is wandering then there is subhorshoes $\Lambda_{i_1}$ and $\Lambda_{i_2}$ such that $\alpha(x)\subset \Lambda_{i_1}$ and $\omega(x)\subset \Lambda_{i_2}$. In particular, there is $z\in \Lambda_{i_1}\cup \Lambda_{i_2}$ such that $\ell_{f,\varphi}(x)=\ell_{f,\varphi}(z)$. If $x$ is nonwandering there is not to do.

%Note that if $\ell=\ell_{f,\varphi}(x)\le t$ then $x\in \Lambda_t$. Therefore $x\in \tilde{\Lambda}_n(t)$, for every $n$. If $x\in \tilde{\Lambda}^i_n(t)$ for some $i<l_n$, that is, $x$ is in a transient set, then, there are $\tilde{\Lambda}^{j_i}_n(t)$ and $\tilde{\Lambda}^{k_i}_n(t)$ with 
%$$x\in W^s(\tilde{\Lambda}^{j_i}_n(t))\cap W^u(\tilde{\Lambda}^{k_i}_n(t)).$$
%Take $z_1\in \tilde{\Lambda}^i_n(t)$ and $z_2\in \tilde{\Lambda}^i_n(t)$ such that $x\in W^s(z_1)\cap W^u(z_2)$ and $l_n\le j_i,k_i\le k_n$.
%Therefore, $\ell_{f,\varphi}(x)=\max\{\limsup_{n\to +\infty}f(\varphi^n(z_1)),\limsup_{n\to -\infty}f(\varphi^n(z_2))\}.$ By compacity, $\ell_{f,\varphi}(x)\in f(\cup_{l_n\le j\le k_n}\tilde{\Lambda}^j_{n}(t)).$
\end{proof}
It follows that
$$HD(L_f(\Lambda)\cap (-\infty, \tilde{a}-\delta)\leq D_n(a-\delta))<1.$$
Therefore,
$$\L(L_f(\Lambda)\cap (-\infty, \tilde{a}-\delta))=\emptyset$$
and we have the Theorem B.

From the above theorem we can ask when do we have $a=\tilde{a}$? To the answer this question we have the following theorem.

\begin{mydef5}\label{TheoremC}
	There is a $C^2$-diffeomorphism $\varphi$ on the sphere, such that there are a $C^2$-open $\mathcal{U}\ni \varphi$ and an open and dense set $\mathcal{U}^{\ast}\subset \mathcal{U}$ such that $a(f,\psi)<\tilde{a}(f,\psi)$ for every $\psi\in \mathcal{U}^{\ast}$ and for a $C^1$-open set of $f:\mathbb{S}^2\rightarrow \mathbb{R}$. In particular,
	$$\inte(M_f(\Lambda_{\psi})\setminus L_f(\Lambda_{\psi}))\neq \emptyset.$$
\end{mydef5}

\begin{proof}
	Consider $\varphi$ with a horseshoe $\Lambda$ with symbolic representation $\{1,2,3,4\}^{\mathbb{Z}}$ and $HD(\Lambda)>1$. We define $f$ using this symbolic representation by 
	$$f((a_n)_n)=\left\{ \begin{array}{rll}
	1, \ \hbox{if} \ a_{-1},a_0\in \{1,2\} \ \hbox{or} \ \hbox{if} & a_{-1},a_0\in \{3,4\}, \\
	1.3, \ \hbox{if} \ a_{-1}\in \{1,2\} \ \hbox{and} \ a_0\in \{3,4\}, \\
	1.6, \ \hbox{if} \ a_{-1}\in \{3,4\} \ \hbox{and} \ a_0\in \{1,2\}. \end{array}\right.$$
Note that we can extend $f$ to a $C^{\infty}$ function on $\mathbb{S}^2$. We take $\varphi:\mathbb{S}^2\rightarrow \mathbb{S}^2$ such that $HD(K^s(\Lambda \{1,2\}))=0.4$, $HD(K^u(\Lambda \{1,2\}))=0.55$, and $HD(K^s(\Lambda \{1,2\}))=0.39$, $HD(K^u(\Lambda \{1,2\}))=0.58$, where $\Lambda\{1,2\}$ is the subhorseshoe with symbolic dynamics given by $\{1,2\}^{\mathbb{Z}}$ and $\Lambda\{3,4\}$ is the subhorseshoe with symbolic dynamics given by $\{3,4\}^{\mathbb{Z}}$. By Lemma \ref{lagrange1}
	$$L_f(\Lambda_{1.3})\subset L_f(\Lambda \{1,2\}\cup \Lambda \{3,4\}).$$ 
Moreover, $L_f(\Lambda \{1,2\}\cup \Lambda \{3,4\})\subset f(\Lambda \{1,2\}\cup \Lambda \{3,4\})$. Since
$$\max\{HD(\Lambda \{1,2\}),HD(\Lambda \{3,4\})\}<1$$ we have that 
	\begin{equation}\label{example}
	\L(L_f(\Lambda)\cap (-\infty,1.3])=0.
	\end{equation}
On the other hand $HD(K^s(\Lambda\{3,4\}))+HD(K^u(\Lambda\{1,2\}))>1$. By Theorem \ref{MR-adapted} there is an open and dense subset of diffeomorphisms such that if $\psi\subset \mathcal{U}^{\ast}$ then 
	$$\inte(M_f(\Lambda_{\psi})\cap (-\infty,1.3)))\neq \emptyset.$$
Moreover, taking $\mathcal{U}^{\ast}$ sufficiently is small, we have that (\ref{example}) is valid for $\psi$, because in this case $\max\{HD(\Lambda_{\psi} \{1,2\}),HD(\Lambda_{\psi} \{3,4\})\}<1$. Note that this is an open property on $f$. In particular, we have that 
	$$a(f,\psi)<1.3<\tilde{a}(f,\psi).$$
	
	\begin{center}
	\begin{figure}
	\begin{tikzpicture}
	\draw (-3,-3) rectangle (3,3);
	\draw (0,0) node{$\Lambda$};
	\draw (-3,-3) rectangle (-1,-1);
	\draw (-2,-2) node{$\Lambda\{1,2\}$};
	\draw (-3,3) rectangle (-1,1);
	\draw (1,1) rectangle (3,3);
	\draw (2,2) node{$\Lambda\{3,4\}$};
        \draw (1,-3) rectangle (3,-1);
         \draw (-2.5,-1.5)--(-2.5,2.5);
        \draw (-2.3,-1.5)--(-2.3,2.5);
        \draw (-2.1,-1.5)--(-2.1,2.5);
         \draw (-2.7,2.4)--(1.3,2.4);
         \draw (-2.7,2.2)--(1.3,2.2);
          \draw (-2.7,2)--(1.3,2);
	\end{tikzpicture}
	\caption{Heteroclinic intersections \label{ferradura}}
	\end{figure}
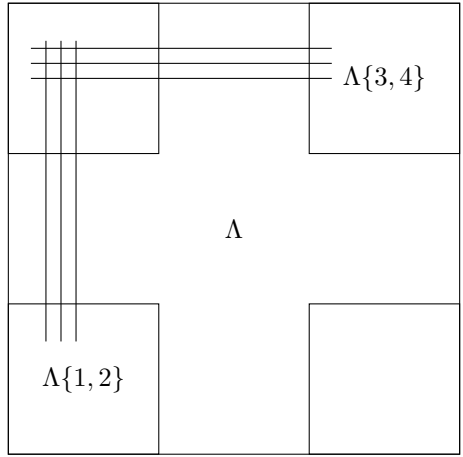
	\end{center}
\end{proof}
\begin{remark}
	We note that the main difference between the Theorems A and B is such that the dimension of stable and unstable Cantor sets can be different. This implies that hyperbolic set of finite type contained at $\Lambda_{a+\delta}$ could have dimension greater than $1$ but every subhorseshoe contained at $\Lambda_{a+\delta}$ has dimension smaller than $1$. This implies that the Markov spectrum can have interior before the Lagrange spectrum. This phenomenon does not occur in the conservative case, where $HD(K^s)=HD(K^u)$.	
\end{remark}

Let $\varphi$ be a conservative diffeomorphism and consider the set 
	$$\mathcal{R}_{\Lambda}'=\{f\in C^k(M,\mathbb{R}); \partial_{s,u}f(x^{s,u})\neq 0\},$$
$k>1$. If $HD(\Lambda')<1$ then $\mathcal{R}_{\Lambda}'$ is an open and dense subset of $C^k(M,\mathbb{R})$, $k>1$. In particular, the set 
	$$\mathcal{R}_{\varphi}=\bigcap_{\Lambda'; HD(\Lambda')<1}\mathcal{R}_{\Lambda}'$$
is a residual subset of $C^k(M,\mathbb{R})$. Using Theorem 1.2 in \cite{CMM16}  there is a residual set $\mathcal{U}^{\ast \ast}$ of diffeomorphisms $\varphi$ such that if $f\in \mathcal{R}_{\varphi}$ then 
	$$HD(L\cap (-\infty, a-\delta))=HD(M\cap (-\infty,a-\delta)).$$
\begin{lemma}
Consider $\varphi\in \mathcal{U}^{\ast \ast}$. If $f\in \mathcal{R}_{\varphi}$ then 
	$$HD(\Lambda \cap f^{-1}(a(f,\varphi)))<1.$$
\end{lemma}

\begin{proof}
We can suppose without loss of generality that $a(f,\varphi)$ is a regular value of $f$ and then dividing the curve in finitely many pieces we can suppose that $f^{-1}(a)$ is formed by curves which are graphs of a $C^k$ function from $W^s(\Lambda)$ to $W^u(\Lambda)$ or from $W^u(\Lambda)$ to $W^u(\Lambda)$ to $W^s(\Lambda)$. In this case, we proceed as the Lemma 17 in \cite{MR2}.

\end{proof}

In particular, if $K^s$ denotes the stable Cantor set associated to $\Lambda$ then 
$$HD(K^s\cap f^{-1}(a))<1/2.$$
 By the Proposition \ref{cons1}, we can apply the above arguments and we have the following theorem. 

\begin{mydef2}
	Given $\varphi\in \Diff^2_{\ast}(M)$ with a horseshoe $\Lambda$, $HD(\Lambda)>1$ and $k\ge 2$, there are an open set $\Diff^2_{\ast}(M)\supset \mathcal{U}\ni \varphi$ and a residual set  $\mathcal{U}^{\ast \ast}\subset \mathcal{U}$ such that for every $\psi\in \mathcal{U}^{\ast \ast}$ there is a $C^k$-residual subset $\mathcal{R}_{\psi}\subset C^k(M,\mathbb{R})$ such that if $f\in \mathcal{R}_{\psi}$ then
	$$\L(M_f(\Lambda_{\psi})\cap (-\infty,a-\delta))= 0 = \L(L_f(\Lambda_{\psi})\cap (-\infty,a-\delta))$$
	and
	$$\inte(L_f(\Lambda_{\psi})\cap (-\infty,a+\delta))\neq \emptyset \neq \inte(M_f(\Lambda_{\psi})\cap (-\infty,a+\delta)).$$
	for all $\delta>0$, where $a=a(f,\psi)$ defined as above.
	Moreover, 
	$$HD(M_f(\Lambda_{\psi})\cap (-\infty,a)))=HD(L_f(\Lambda_{\psi})\cap (-\infty,a))=1.$$
\end{mydef2} 

\begin{proof} By simplicity, we just write $\Lambda$ for $\Lambda_{\psi}$.
Since $$HD(M_f(\Lambda)\cap (-\infty,a-\delta))=HD(L_f(\Lambda)\cap (-\infty,a-\delta))\le HD(f(\Lambda_{a-\delta}))\le HD(\Lambda_{a-\delta})<1$$ we have that
$$\L(M_f(\Lambda_{\psi})\cap (-\infty,a-\delta))= 0 = \L(L_f(\Lambda_{\psi})\cap (-\infty,a-\delta)).$$
We need to proof only that 
	$$HD(M_f(\Lambda_{\psi})\cap (-\infty,a)))=HD(L_f(\Lambda_{\psi})\cap (-\infty,a))=1.$$ Suppose that $\lim_{\delta \to 0}HD(\Lambda_{a-\delta})=d_1<1$. In particular $HD(K^s_{a-\delta})<1/2$ for all $\delta>0$.%We know that if $t<a$ then, see \cite{CMM16},
	% $$HD(L_{f}(\Lambda)\cap (-\infty,t))=HD(M_f(\Lambda)\cap (-\infty,t))=HD(\Lambda_t).$$ 

 Given $\epsilon>0$, let $\tilde{d}=HD(K^s\cap f^{-1}(a))$. 
Following \cite{CMM16} if we have a Markov partition $\mathcal{P}$, consider 
$$K^s_{t}=\bigcup_a \pi^s(\Lambda_t\cap P_a).$$

We can take a cover of $K^s\cap f^{-1}(a)$ by $\epsilon$-intervals $\{I^s_{\underline{a}^i}\}_{i=1}^M$	 such that $M=\Oh(\epsilon^{-\tilde{d}})$. If $\delta>0$ is sufficiently small we have that 
$$K^s_{a+\delta}\subset K^s_{a-\delta}\cup (\cup_{i=1}^M I^s_{\underline{a}^i} ).$$
There is $m>1$ such that we can find a cover of $K^s_{a-\delta}$ by $\epsilon^m$-intervals $\{J_{\underline{b}^k}\}_{k=1}^N$, where $N=\Oh(\epsilon^{-md})$, where $d=\max\{\tilde{d},HD(K^s_{a-\delta})\}<1/2$. In particular, $\cup_{k=1}^N J_{\underline{b}^k}\cup \cup_{i=1}^M I_{\underline{a}^i}$ is an $\epsilon$-cover of $K_{a+\delta}$. Note that,
$$\sum_{k=1}^M |I^s_{\underline{b}^k}|^d+\sum_{i=1}^N |I^s_{\underline{a}^i}|^d=\Oh(1),$$
then $HD(K^s_{a+\delta})\le d<1/2$. Since $HD(\Lambda_{a+\delta})=2\cdot HD(K^s_{a+\delta})<1$ we have a contradiction with the definition of $a$. Then 
$$\lim_{\delta \to 0} HD(L_f(\Lambda)\cap (-\infty,a-\delta))=1=\lim_{\delta \to 0} HD(M_f(\Lambda)\cap (-\infty,a-\delta)).$$

\end{proof}

 Then, the proof of theorem D, gives us that, in this context, $L_f(\Lambda)\cap (a-\delta,a)$ is a countable union of Cantor sets.
The above Theorem inspires us to compare the Dynamical with the Classical Spectra.

Let us consider $A_N=[0;\overline{N,1}]$ and $B_N=[0;\overline{1,N}]$. We have that 
$$NA_N+a_NB_N=1 \ \mbox{and} \ B_N+B_NA_N=1,$$ 
so
$A_N=\dfrac{B_N}{N}.$

Thus, we have

$$B_N=\frac{-N+\sqrt{N^2+4N}}{2} \ \mbox{and} \ A_N=\frac{-N+\sqrt{N^2+4N}}{2N}.$$ 
We denote $C(N)=\{x=[0;a_1,a_2,...]; a_i\le N, \forall i\ge 1\}$. Let $g$ be the Gauss map, $g(x)=\{1/x\}$, where $\{z\}=z-\lfloor z \rfloor$. We write $\tilde{C}(N)=\{1,2,...,N\}+C(N)$, $\Lambda_N=C(N)\times \tilde{C}(N)$ and $f(x,y)=x+y$. If $x=[0;a_1(x),a_2(x),...]$ then take $\varphi:\Lambda_N \rightarrow \Lambda_N$ given by
$$\varphi(x,y)=(g(x),a_1(x)+1/y).$$

We note that $\varphi$ can be extended to a $C^2$-conservative diffeomorphism. If fact, $\tilde{\varphi}(x,y)=\left(g(x),\dfrac{1}{a_1(x)+y}\right)$ defined on $C(N)\times C(N)$ has an invariante measure, see \cite{S.ITO}, and $h(x,y)=(x,1/y)$ conjugates $\varphi$ with $\tilde{\varphi}$. Moreover, $\max f(\Lambda_N)=2B_N+N=\sqrt{N^2+4N}.$ In particular, $\max f(\Lambda_4)=\sqrt{32}$.
It is so easy to show that

\begin{proposition}
	If $L$ is the classical Lagrange spectrum, then 
	$$L\cap (-\infty,1+\sqrt{21}]=L_f(\Lambda_4)\cap (-\infty,1+\sqrt{21}].$$
	Moreover, if $M$ is the classical Markov spectrum, then
		$$M\cap (-\infty,1+\sqrt{21}]=M_f(\Lambda_4)\cap (-\infty,1+\sqrt{21}].$$

\end{proposition}

\begin{proof}
	Given any sequence $\underline{\theta}=(a_n)_{n\in \mathbb{Z}}\in \mathbb{N}^{\mathbb{Z}}$ let us consider the set $N_j(\underline{\theta})=\{n\in \mathbb{N}; a_n\ge j\}$. It is easy to see that if $N_j(\underline{\theta})$ is infinite then
	
	$$\ell(\underline{\theta})\ge j+2A_j$$
	Therefore, if we have a sequence $\underline{\theta}$ such that $N_5(\theta)$ is infinite we have that $\ell(\theta)\ge \frac{20+\sqrt{45}}{5}=5.3416407865....$
	This implies that 
	$$L_f(\Lambda_4)\cap (-\infty, 5.34]=L\cap (-\infty, 5.34].$$

	If we take $\ell\in L\cap (-\infty, 1+\sqrt{21}]$ such that $\ell>5.34$, by the theorem 5, chapter 4 of [3], we have that $l=4+[0;a_1,a_2,...]+[0;a_{-1},a_{-2},...]$, where $(a_i,a_{i+1})\notin \{(1,4),(2,4)\}$, $i\in \mathbb{N}$ and $(a_{-i},a_{-(i+1)})\notin \{(1,4),(2,4)\}$, $i\in \mathbb{N}$. Moreover, $f(\sigma^k(\underline{\theta}))\le f(\underline{\theta})$, 
	where we define $\underline{\theta}=(a_n)_{n\in \mathbb{Z}}$, with $a_0=4$. Then we have $f(\underline{\theta})=\ell$. Let us consider $$\underline{\theta}^m=(a_{-m},a_{-m+1},...,a_{-1},a_0,a_1,...,a_{m-1},a_m), m\in \mathbb{N}.$$ So, if $\underline{\alpha}=(\underline{\theta}^{|m|})_{m\in \mathbb{Z}}$ we can find a sequence $x_n\in \mathbb{N}$ such that 
	$$\lim_{n\to +\infty}f(\sigma^{x_n}(\underline{\alpha}))=f(\underline{\theta})=\ell,$$
	$x_n=n(n+1)$ for instance.
	
	So, given $\ell\in L\cap (-\infty, 1+\sqrt{21}]$ we can find $\underline{\alpha}\in \{1,2,3,4\}^{\mathbb{Z}}$ such that $\ell(\underline{\alpha})=\ell$, in particular, $L_f(\Lambda_4)\cap (-\infty,1+\sqrt{21}]=L\cap (-\infty, 1+\sqrt{21}]$. The proof that
	$$M\cap (-\infty,1+\sqrt{21}]=M_f(\Lambda_4)\cap (-\infty,1+\sqrt{21}]$$
	is analogous.

\end{proof}

A little bit more is true, in fact $L\cap (-\infty, \sqrt{32}]=L_f(\Lambda_4)$, but it is not so easy to show.

From this, we hope that the \textbf{question 1}  by Moreira is true:

\begin{conj} Let $L$ be the classical Lagrange spectrum and $M$ be the classical Markov spectrum. Consider $$d(t)=HD(L\cap (-\infty,t))=HD(M\cap (-\infty,t)) \ \mbox{and} \ a=\inf\{t\in \mathbb{R}; d(t)=1\}.$$ Then
	$$\L(L\cap (-\infty,a-\delta))=0=\L(M\cap (-\infty,a+\delta))$$
	but
	$$\inte(L\cap (-\infty,a+\delta))\neq \emptyset \ \mbox{and} \ \inte(M\cap (-\infty,t+\delta))\neq \emptyset$$
	for all $\delta>0$.
\end{conj}
If the above conjecture is true we have the following striking result:
	$$\inte(C(2)+C(2))\neq \emptyset,$$
where $C(2)+C(2)=\{x+y;x,y\in C(2)\}.$

%where $[g_{|I_{\lfloor x\rfloor}}]^{-1}(y)$ is the inverse of Gauss map on the interval $\left(\frac{1}{\lfloor x\rfloor+1},\frac{1}{\lfloor x\rfloor}\right].$ 
Recall that $HD(\tilde{C}(2)\times C(2))=HD(C(2))+HD(C(2))>1$. Moreover, if we define $f:U\rightarrow \mathbb{R}$ by 
	$$f(x,y)=x+y$$
then $L\cap (-\infty,c)\supset L_f(\Lambda_2)$, observe that $\max L_f(\Lambda_2)=\max f(\Lambda_2)=2+2B_2=\sqrt{12}$. On the other hand, by remark \ref{r3.1} we have that 
$$L\cap (-\infty,a+\delta)\subset f(\Lambda_{a+\delta}).$$
Since $a<\sqrt{12}$ (cf. \cite{M3}) we must have for $\delta>0$ suficiently small that 
$f(\Lambda_{a+\delta})\subset f(\Lambda_2)=\tilde{C}(2)+C(2).$ So, if the conjecture is true, $$\emptyset \neq \inte(L\cap (-\infty, a+\delta))\subset \inte(\tilde{C}(2)+C(2)).$$
Thus, $\inte(C(2)+C(2))\neq \emptyset$.  

\begin{remark}
	We observe for that the above conjecture be true is sufficient that $\varphi\in \mathcal{R}$, i.e., $\varphi$ can not has subhorseshoes with Hausdorff dimension equal to $1$ and for every subhorseshoe with Hausdorff dimension greater than $1$, in this way we have $\inte{L_f(\tilde{\Lambda})}\neq \emptyset$, where $\varphi$ and $f$ as same above. We note that in this case we have $f\in \mathcal{H}_{\varphi}$.
\end{remark}
\section{The beginning of the spectra}

Next we prove a theorem about the beginning of the Markov and Lagrange dynamical spectra for all $C^2$-diffeomorphism $\varphi$ having a horseshoe $\Lambda$ and the most $C^1$-maps $f:M\rightarrow \mathbb{R}$. %Moreover, we leave an interesting question of J. C. Yoccoz. 
Let $b(f,\varphi)=\inf\{t\in \mathbb{R};HD(\Lambda_t)>0\}$. Note that $b(f,\varphi)<+\infty$. Let $B=\{\delta\in (0,+\infty);f^{-1}(b+\delta)\cap \Lambda\neq \emptyset\}$. Note that $\Lambda_b$ cannot be a subhorseshoe, because otherwise we could find another subhorseshoe with smaller Hausdorff dimension than it, in a similary way such as lemma 6 of \cite{MR-15}. Similarly to the proposition \ref{a1}, $B$ satisfies the following propositions

\begin{proposition}\label{B1}
	$0$ is an accumulation point of $B$.
\end{proposition}

\begin{proof}
	The proof is analogous to the proof of Proposition \ref{a1}. Otherwise $\Lambda_b=\Lambda_{b+\delta}$ for some $\delta$. But this implies that $\Lambda_b$ is a subhorseshoe, contradiction.
\end{proof}

\begin{proposition}\label{B2}
	Given $\delta>0$ there is a subhorseshoe $\tilde{\Lambda}$ with
	$$\Lambda_b\subsetneq \tilde{\Lambda}\subset \Lambda_{b+\delta}.$$
\end{proposition}
\begin{proof}
	Essentially the same proof of the proposition \ref{p1}, but now we take take $\delta\in B$. We only observe that in this present case we do not need of a residual subset $\mathcal{R}$ because there is no subhorseshoe with zero Hausdorff dimension.
\end{proof}

\begin{lemma}\label{le1}
	Let $K$ and $K'$ be two $C^s$-regular Cantor sets, $s>1$ such that $K\times K'\subset U$, where $U$ is an open set in $\mathbb{R}^2$. Suppose that $f:U\rightarrow \mathbb{R}$ be a $C^1$ function and there is $(x_0,y_0)\in K\times K'$ such that 
	%$$\frac{\partial f(x_0,y_0)}{\partial x}\neq 0$$ and 
	$$\frac{\partial f(x_0,y_0)}{\partial y}\neq 0.$$
	Then,
	$$HD(f(K\times K'))>0.$$
\end{lemma}
\begin{proof}
	Since $K\times K'$ is a compact set and $(x_0,y_0)\in K\times K'$ we can take an interval $I$ such that $(x_0,y)\in \{x_0\}\times I\subset U$ and $I$ is an interval of the construction of $K'$. By hypoteses $g_{x_0}:I\rightarrow \mathbb{R}$ defined by $g_{x_0}(y)=f(x_0,y)$ is a diffeomorphism on a possible small interval $J\subset I$. Hence $HD(g_{x_0}(J\cap K'))>0$ because $J\cap K'$ contains an interval of the construction of $K'$ and diffeomophisms preserve Hausdorff dimension. But $g_{x_0}(J\cap K')\subset f(K\times K')$. Therefore,
	$$HD(f(K\times K'))>0.$$
	Now we are ready to show
\end{proof}
\begin{mydef6}
	If $\varphi\in \Diff^2(M)$ has a horseshoe $\Lambda$ and $k\ge 1$ is given, then there is a residual subset $\mathcal{H}_{\varphi}\subset C^k(M,\mathbb{R})$ such that for every $f\in \mathcal{H}_{\varphi}$
	$$HD(L_f(\Lambda)\cap (-\infty,b+\delta))>0$$
	and
	$$HD(M_f(\Lambda\cap (-\infty,b+\delta)))>0,$$
	for all $\delta>0$, where $b=b(f,\varphi)=\inf\{t\in \mathbb{R};HD(\Lambda_t)>0\}$.
\end{mydef6}
\begin{proof}
	We denote by $S(\Lambda)=\{\Lambda^1,\Lambda^2,...\}$ the set of hyperbolic sets of finite type of $\Lambda$. We know that $\Lambda^i=\Lambda^{i,1}\cup ...\cup \Lambda^{i,n_i}$. We set $H_{\varphi}(\Lambda^i)=H_{\varphi}(\Lambda^{i,n_i})$, where $\Lambda^{i,n_i}$ is a horseshoe . Note that if $\Lambda^i$ is a horseshoe then $n_i=1$. Recall that, see \cite{MR-15}, every $f\in H_{\varphi}(\Lambda^i)$ has a unique maximum point $x$ such that $Df(x)e^{s,u}_x\neq 0.$ %In this present case is not important $\Lambda^{i,n_i}$ be the horseshoe with maximal dimension.
	 In \cite{MR-15} has been shown that if $f\in H_{\varphi}(\Lambda^{i,n_i})$ there is $\Lambda^j\subset \Lambda^{i,n_i}$ such that $f(\tilde{B}(\Lambda^j))\subset L_f(\Lambda^{i,n_i})\subset L_f(\Lambda)$, where $\tilde{B}$ is a diffeomorphism with $D(f\circ \tilde{B})(x)e^{s,u}_x\neq 0$. By the definition of $b$ we have that $HD(\Lambda_{b+\delta})>0$. Then, by proposition \ref{B2} there is a hyperbolic set of finite type $\tilde{\Lambda}$ with
	$$\Lambda_b\subset \tilde{\Lambda}\subset \Lambda_{b+\delta}.$$ Suppose then $f\in \mathcal{H}_{\varphi}=\bigcap_jH_{\varphi}(\Lambda^j)$.
	So, we can find another subhorseshoe $\Lambda'\subset \tilde{\Lambda}$ and a diffeomorphism $\tilde{B}$ such that
	\begin{equation}\label{eqtb1}
	f(\tilde{B}(\Lambda'))\subset L_f(\tilde{\Lambda})\subset L_f(\Lambda_{b+\delta}).
	\end{equation}
	But, by remark \ref{r3.2}, we have that  $L_f(\Lambda_{b+\delta})\subset L_f(\Lambda)\cap (-\infty,b+\delta]$. 
	Then, we have showed that
	$$f(\tilde{B}(\Lambda'))\subset L_f(\Lambda)\cap (-\infty,b+\delta].$$
	But $D(f\circ \tilde{B})(x)e^{s,u}_x\neq 0$. Therefore
	$$HD(f\circ \tilde{B}(\Lambda'))>0.$$%A deeper dimension formula, due to Moreira \cite{M2}, says that in 
	%this case $HD(f(\Lambda'))=HD(\Lambda')>0$. Since $\delta$ was taken arbitrarily we have showed that
	Then,
	
	$$HD(L_f(\Lambda)\cap (-\infty, b+\delta))>0.$$
	In particular, this also prove that
	
	$$HD(M_f(\Lambda)\cap (-\infty,b+\delta))>0.$$  
	
	%[Eu sei que n\~ao preciso usar a f\'ormula da dimens\~ao, escrevi apenas pra deixar uma prova escrita].
	
	%same arguments as the above section we can find $\tilde{\delta}<\delta$ such that $f^{-1}(b+{\tilde{\delta}})$ by elements of some Markov partition $\mathcal{P}$, with
	%$$f^{-1}((-\infty,b+\delta))\cap (\cup_{P\in \mathcal{P}}P)\subsetneq f^{-1}((-\infty,b+\delta)).$$
	%So there is a subhorseshoe $\tilde{\Lambda}$ between $\Lambda_b$ and $\Lambda_{b+\delta}.$
\end{proof}

\begin{remark}
	We note that the above theorems does not occur in all the cases. The most simple example is to take $f$ any constant function.
\end{remark}

\textbf{Acknowledgements}: We would like to thank Sergio Roma\~na for very valuable discussions.

%\bibliographystyle{amsplain}
%\begin{thebibliography}{10}

%\bibitem {A} T. Aoki, \textit{Calcul exponentiel des op\'erateurs
%microdifferentiels d'ordre infini.} I, Ann. Inst. Fourier (Grenoble)
%\textbf{33} (1983), 227--250.

%\bibitem {B} R. Brown, \textit{On a conjecture of Dirichlet},
%Amer. Math. Soc., Providence, RI, 1993.

%\bibitem {D} R. A. DeVore, \textit{Approximation of functions},
%Proc. Sympos. Appl. Math., vol. 36,
%Amer. Math. Soc., Providence, RI, 1986, pp. 34--56.

%\end{thebibliography}

\end{document}